\documentclass[11pt]{amsart}

\usepackage{amscd,amssymb,amsopn,amsmath,amsthm,mathrsfs,graphics,amsfonts,enumerate,verbatim,calc
}
\usepackage[all]{xypic}

\usepackage[OT2,OT1]{fontenc}
\newcommand\cyr{%
\renewcommand\rmdefault{wncyr}%
\renewcommand\sfdefault{wncyss}%
\renewcommand\encodingdefault{OT2}%
\normalfont
\selectfont}
\DeclareTextFontCommand{\textcyr}{\cyr} 

\usepackage{amssymb,amsmath}

\DeclareFontFamily{OT1}{rsfs}{}
\DeclareFontShape{OT1}{rsfs}{n}{it}{<-> rsfs10}{}
\DeclareMathAlphabet{\mathscr}{OT1}{rsfs}{n}{it}

\numberwithin{equation}{section}
\newtheorem{theorem}{Theorem}[section]
\newtheorem{lemma}[theorem]{Lemma}
\newtheorem{proposition}[theorem]{Proposition}
\newtheorem{corollary}[theorem]{Corollary}

\newtheorem{maintheorem}{Main Theorem}

\theoremstyle{definition}
\newtheorem{definition}[theorem]{Definition}
\newtheorem{remark}[theorem]{Remark}
\theoremstyle{remark}

\newtheorem{example}[theorem]{Example}
\newtheorem{acknowledgement}{Acknowledgement}

\newcommand{\im}{\operatorname{Im}}
\renewcommand{\ker}{\operatorname{Ker}}

\newcommand{\Spec}{\operatorname{Spec}}

\newcommand{\Tor}{\operatorname{Tor}}
\newcommand{\Hom}{\operatorname{Hom}}

\newcommand{\coker}{\operatorname{Coker}}

\newcommand{\fm}{\frak{m}}

\begin{document}
\title[Finite \'etale extensions of Tate rings and decompletion of perfectoid algebras]
{Finite \'etale extensions of Tate rings and decompletion of perfectoid algebras}

\author[K.Nakazato]{Kei Nakazato}
\address{Graduate School of Mathematics, Nagoya University, 
Nagoya 464-8602, Japan}
\email{m11047c@math.nagoya-u.ac.jp}

\author[K.Shimomoto]{Kazuma Shimomoto}
\address{Department of Mathematics, College of Humanities and Sciences, Nihon University, Setagaya-ku, Tokyo 156-8550, Japan}
\email{shimomotokazuma@gmail.com}




\begin{abstract}
In this paper, we examine the behavior of ideal-adic separatedness and completeness under certain ring extensions using trace map. Then we prove that adic completeness of a base ring is hereditary to its ring extension under reasonable conditions. We aim to give many results on ascent and descent of certain ring theoretic properties under completion. As an application, we give conceptual details to the proof of the almost purity theorem for Witt-perfect rings by Davis and Kedlaya. Witt-perfect rings have the advantage that one does not need to assume that the rings are complete and separated.
\end{abstract}

\maketitle

\tableofcontents

\section{Introduction}

In the basic part of perfectoid geometry, one of the most fundamental tools that is frequently used is the \textit{Almost purity theorem}, which was first proved by Faltings for certain big algebras constructed from smooth algebras over a discrete valuation ring, and then by Scholze for perfectoid algebras over a perfectoid field. One drawback of perfectoid rings is that one needs to work with $p$-adically complete rings, which prevents us from taking infinite integral extensions of \textit{$p$-adically complete} rings directly. Roughly speaking, a \textit{Witt-perfect} condition on a $p$-torsion free ring is defined in the same way as for perfectoid algebras, except that it need not be $p$-adically complete. This class of rings had been introduced by Davis and Kedlaya in \cite{DK14} and \cite{DK15}. However, a difficulty lurks in dealing with Witt-perfect rings, due to the lack of tilting correspondence for those rings. 

The present note stems from authors' endeavor to reaching deeper understanding of the papers \cite{DK14} and \cite{DK15}. Our aim is to prove some basic results and explain their consequences which are of great importance in the situation where one wants to avoid taking $p$-adic completion. We also make extensive studies on the comparisons between (almost) Witt-perfect algebra and (almost) perfectoid algebras. It is clear from Andr\'e's work \cite{An1} that almost perfectoid algebras naturally show up in certain applications.

$\bf{Notation}$: All rings are assumed to be commutative with a unity. For a finite projective ring extension $A \hookrightarrow B$, denote by $\textnormal{Tr}_{B/A}:B \to A$ the trace map; see the book \cite{Fo17} for the construction. For a ring $A$ and an $A$-algebra $B$, we denote by $A^+_B$ (resp.\ $A^*_B$) the $A$-subalgebra of $B$ consisting of all elements that are integral (resp.\ almost integral) over $A$, which are found in Definition \ref{Defintcompleteint}. For the definition of a Tate ring $A$ with the subset of powerbounded elements $A^\circ$, see Definition \ref{TateDefinition}.

Let us state the main results; see Proposition \ref{prop513}, Proposition \ref{prop1224}, Theorem \ref{maincor1} and Theorem \ref{lemalprjct} below.

\begin{maintheorem}
Let $A$ be a ring and let $B$ be a finite \'etale $A$-algebra. Let $A_{0}\subset A$ be a subring with an ideal $I_{0}\subset A_{0}$. Let $B_{0}$ be an $A_{0}$-subalgebra of $B$ such that $B=\bigcup_{n\geq 1}(B_{0}:_B I_{0}^{n})$ (see $(\ref{colonideal})$ for this notation below). Assume that there exists an integer $c>0$ such that $\textnormal{Tr}_{B/A}(tm)\in A_{0}$ for every $t\in I_0^c$ and every $m\in B_{0}$. 

\begin{enumerate}
\item
If $A_{0}$ is $I_{0}$-adically separated, then so is $B_{0}$. 

\item
If $A_{0}$ is $I_{0}$-adically complete, then so is $B_{0}$. 
\end{enumerate}
\end{maintheorem}

\begin{maintheorem}
Let $(R, I)$ be a basic setup and let $f_0: A_0\to B_0$ be an $R$-algebra homomorphism with an element $t\in A_0$. Denote by $\widehat{A_0}$ and $\widehat{B_0}$ the $t$-adic completions of $A_0$ and $B_0$, respectively. Let $\widehat{f_0}: \widehat{A_0}\to \widehat{B_0}$ be the $R$-algebra homomorphism induced by $f_0$. Assume that the morphism of pairs $f_0: (A_0, (t))\to (B_0, (t))$ satisfies condition $(*)$ (cf.\ Definition \ref{CondQuad*}). Then the following assertions hold. 
\begin{enumerate}
\item
The natural $\widehat{A_0}[\frac{1}{t}]$-algebra homomorphism $B_0[\frac{1}{t}]\otimes_{A_0[\frac{1}{t}]}\widehat{A_0}[\frac{1}{t}] \to \widehat{B_0}[\frac{1}{t}]$ is an isomorphism. 

\item
$\widehat{f_0}: (\widehat{A_0}, (t))\to (\widehat{B_0}, (t))$ also satisfies the condition $(*)$ (cf.\ Definition \ref{CondQuad*}).

\item
The following conditions are equivalent. 
\begin{itemize}
\item[$(a)$]
$B_0$ is $I$-almost finitely generated and $I$-almost projective over $A_0$. 
\item[$(b)$]
$\widehat{B_0}$ is $I$-almost finitely generated and $I$-almost projective over $\widehat{A_0}$. 
\end{itemize}
\item
The following conditions are equivalent. 

\begin{itemize}
\item[$(a)$]
$f_0: A_0\to B_0$ is $I$-almost finite \'etale. 
\item[$(b)$]
$\widehat{f_0}: \widehat{A_0}\to \widehat{B_0}$ is $I$-almost finite \'etale. 
\end{itemize}
\end{enumerate}
\end{maintheorem}

An important consequence that follows from the above results is the almost purity theorem for Witt-perfect rings; see Theorem \ref{almostpurity}. This statement was originally established by Davis and Kedlaya which ultimately relies on the almost purity theorem by Kedlaya-Liu; see papers \cite{DK14} and \cite{DK15} and \cite{KL15}.

\begin{corollary}[Almost purity]
Let $A_0$ be a $p$-torsion free Witt-perfect ring and let $f_0: A_0\to B_0$ be a ring map. Put $I:=\sqrt{pA_0}$. Assume that the morphism $f_0: (A_0, (p))\to (B_0, (p))$ satisfies condition $(*)$ (cf.\ Definition \ref{CondQuad*}), $A_0$ is integrally closed in $A_0[\frac{1}{p}]$, $(A_0)^+_{B_0[\frac{1}{p}]}\subset B_0$, and $(A_0, I)$ is a basic setup. Then the following assertions hold.\begin{enumerate}
\item
$B_0$ is also Witt-perfect.

\item
$f_0: A_0\to B_0$ is $I$-almost finite \'etale.
\end{enumerate}
\end{corollary}

The structure of this paper goes as follows.

In \S 2, we give a review on Tate rings, a reasonable class of topological rings that are introduced by Huber \cite[p.\ 456]{Hu93} as a generalization of classical Tate algebras. This class is useful for describing the structure as a topological ring of a Banach ring, and developing an algebraic theory approximate to non-archimedean analysis. As a preliminary to the study of uniform Banach rings, we introduce the notion of \textit{preuniformity} for both Tate rings and pairs of rings $(A,I)$, where $A$ is a ring and $I$ is an ideal of $A$; see Definition \ref{def204} and Definition \ref{defpair}. This notion is important when one needs to distinguish uniform Banach rings and uniform non-Banach rings and moreover, our definition is of algebraic nature that is appealing to algebraists. Thus, their properties are studied with connections to Banach rings. 

In \S 3, we will use Fontaine's version of perfectoid rings that was introduced in \cite{Fo13} as a generalization of Scholze's perfectoid algebras; see Definition \ref{Fontainedef}. We study its relation to Witt-perfect rings, which are decompletions of perfectoid algebras.

In \S 4, we study finite \'etale extensions over a Tate ring. In particular, the reader will find that a \textit{trace map} becomes an essential tool for testing certain topological structure such as complete or separated, on module-finite algebras over Tate rings. The main results in this section are Proposition \ref{prop513}, Proposition \ref{prop1224} and Theorem \ref{maincor1}. We also prove some basic results on (complete) integral closure and their behavior under completion. We believe that some of these results are known to experts. However, as they do not seem to be documented in existing literatures, we try to give detailed proofs with maximal generality.

In \S 5, by specializing the results obtained in \S 3, we complete the proof of the almost purity theorem for Witt-perfect rings by reducing to the case of perfectoid rings by Kedlaya-Liu \cite{KL15}; see Theorem \ref{lemalprjct}
and Theorem \ref{almostpurity}. A review on some basic definitions of almost ring theory is also given, borrowing from Gabber-Ramero's monograph \cite{GR03}.

In \S 6, we give some historical remarks on the almost purity theorems. The reason for the inclusion of this appendix is to help commutative algebraists to understand core ideas. No new results are proved here.

Let us point out that the almost purity theorem without completion has also been established by Gabber-Ramero in \cite{GR18}, under the name of \textit{formal perfectoid rings}. Although their treatment is quite general, we think that our approach is short and concise. We plan to apply the results of this paper to the construction of almost Cohen-Macaulay algebras after establishing a variant of Andr\'e's perfectoid Abhyankar's lemma in the forthcoming paper \cite{NS19}.

\section{Notation and preliminaries}

For the definition of perfectoid algebras, we follow the original version by Scholze \cite{Sch12} and its essential extension by Fontaine \cite{Fo13}. There is, however, a more general version as introduced in the paper \cite{BMS17}. A \emph{pair} is meant to be a pair $(A, I)$ consisting of a ring $A$ and an ideal $I\subset A$. 
A \emph{morphism of pairs} $f: (A, I)\to (B, J)$ is a ring map $f: A\to B$ such that $I^n\subset f^{-1}(J)$ for some $n>0$.  
Let $(A, I)$ be a pair, $M$ be an $A$-module, and $N$ be an $A$-submodule of $M$. Then the \emph{$I$-saturation of $N$ in $M$} is defined to be the $A$-submodule 
$$
N^{I\textnormal{-sat}}=\{m\in M\ |\ \textnormal{for any }x\in I,\ \textnormal{there exists some } n>0\ \textnormal{such that }x^nm\in N \}. 
$$
In particular, for the ideal $(0)\subset A$, the $I$-saturation $(0)^{I\textnormal{-sat}}$ denotes the ideal of $A$ consisting of all $I$-torsion elements.

\subsection{Integrality and almost integrality}
Here we discuss the notion of integrality and almost integrality.\footnote{The notion of almost integrality is much older than $I$-almost integrality, which is relevant to the conclusion of Lemma \ref{lemma117}. This nomenclature is somehow a mathematical incident in a positive sense. Indeed, one of our purposes is to relate almost integrality to various notions used in almost ring theory.} A general reference is \cite{Bour98}.

\begin{definition}
\label{Defintcompleteint}
Let $A \subset B$ be a ring extension. 
\begin{enumerate}
\item
An element $b \in B$ is \textit{integral} over $A$, if $\sum_{n=0}^\infty A \cdot b^n$ is a finitely generated $A$-submodule of $B$. The set of all elements, denoted as $C$, of $B$ that are integral over $A$ forms an $A$-subalgebra of $B$. If $A=C$, then $A$ is called \textit{integrally closed} in $B$. 

\item
An element $b \in B$ is \textit{almost integral} over $A$, if $\sum_{n=0}^\infty A \cdot b^n$ is contained in a finitely generated $A$-submodule of $B$. The set of all elements, denoted as $C$, of $B$ that are almost integral over $A$ forms an $A$-subalgebra of $B$, which is called the \textit{complete integral closure} of $A$ in $B$. If $A=C$, then $A$ is called \textit{completely integrally closed} in $B$. 
\end{enumerate}
\end{definition}

This definition can be extended to any ring homomorphism $A \to B$ in a natural way: Let $A$ be a ring, let $B$ be an $A$-algebra and let $b\in B$ be an element. Then  we say that $b$ is \textit{integral} (resp.\ \textit{almost integral}) over $A$, if $b$ is integral (resp.\ almost integral) over the image of $A$ in $B$. Unlike integral closure, the complete integral closure of an integral domain in its field of fractions is not necessarily completely integrally closed in the same field of fractions; see \cite{He69} for such examples.

In the case when $A$ is Noetherian, we have $A^{+}_{B}=A^{*}_B$. However, these two rings may be quite different in general. 

\begin{example}\label{int.neq.aint}
Let $G$ be the abelian group $\mathbb{Z}\oplus \mathbb{Z}$ equipped with the lexicographic order (i.e.\ $(a_{1}, a_{2})\leq (b_{1}, b_{2})\iff$ $a_{1}<b_{1}$, or $a_{1}=b_{1}$ and $a_{2}\leq b_{2}$). Let $V$ be a valuation ring with value group $G$, $K=\textnormal{Frac}(V)$, and let $v: K^{\times}\to G$ be a valuation corresponding to $V$. Let $t\in V$ be an element such that $v(t)=(1,0)$. Then $K=V[\frac{1}{t}]$. Hence $a\in K$ is almost integral over $V$ if and only if there exists some $l>0$ such that $t^{l}a^{n}\in V$ for every $n>0$.  For $a\in K\setminus\{0\}$, the latter condition means that the first entry $a_{1}$ of $v(a)$ satisfies $a_{1}\geq 0$. Hence $V^*_{K}$ is a valuation ring of rank $1$. On the other hand, since $V^{+}_{K}=V$, $V^{+}_{K}$ is a valuation ring of rank $2$. Thus the Krull dimension of $V^{+}_{K}$ and that of $V^{*}_{K}$ are different. 
\end{example}

We often use the following results.

\begin{lemma}\label{AlmIntMod-Fin}
Let $A_0$ be a ring with a nonzero divisor $t$. Let $B_0$ be a $t$-torsion free $A_0$-algebra such that the induced $A_0[\frac{1}{t}]$-algebra $ B_0[\frac{1}{t}]$ is module-finite. Then one has $t(A_0)^*_B\subset (A_0)^+_B$. 
\end{lemma}

\begin{proof}
Put $A:=A_0[\frac{1}{t}]$, $B:=B_0[\frac{1}{t}]$ and $B'_0:=(A_0)^{+}_{B}$. 
Pick $b\in (A_0)^*_B$. Then there exists a finitely generated $A_0$-submodule $N_0\subset B$ such that $b^n$ belongs to $N_0$ for every $n>0$. On the other hand, since $B$ is module-finite over $A$, we have $B=B'_0[\frac{1}{t}]$. Thus, $t^lN_0$ is contained in $B'_0$ for some $l>0$. In particular, $(tb)^l\ (=t^lb^l)$ lies in $B'_0$ and therefore, so does $tb$ by the definition of $B'_0$. Hence $tb$ is integral over $A_0$, as desired. 
\end{proof}

\begin{proposition}
\label{CICprop}
Let $A_0$ be a ring with a nonzero divisor $t$. Denote by $\widehat{A_0}$ the $t$-adic completion of $A_0$. Put $A:=A_0[\frac{1}{t}]$,  $A':=\widehat{A_0}[\frac{1}{t}]$, $A^+:=(A_0)^+_A$, and $A^\circ:=(A_0)^*_A$. 
\begin{enumerate}
\item
Suppose that there exists some $c\geq 0$ for which $t^cA^+\subset A_0$ (resp. $t^cA^\circ\subset A_0$). Then the following assertions hold. 
\begin{enumerate}
\item
One has 
$t^c(\widehat{A_0})^+_{A'}\subset \widehat{A_0}$ (resp. $t^c(\widehat{A_0})^*_{A'}\subset \widehat{A_0}$).

\item
Denote by $\widehat{A^+}$ and $\widehat{A^\circ}$ the $t$-adic completions. Then 
the inclusion map $A_0\hookrightarrow A^+$ (resp.\ $A_0\hookrightarrow A^\circ$) induces an isomorphism $A'\xrightarrow{\cong} \widehat{A^+}[\frac{1}{t}]$ 
(resp.\ $A'\xrightarrow{\cong} \widehat{A^\circ}[\frac{1}{t}]$) whose restriction to $(\widehat{A_0})^+_{A'}$ (resp.\ $(\widehat{A_0})^*_{A'}$) yields an isomorphism $(\widehat{A_0})^+_{A'}\xrightarrow{\cong}\widehat{A^+}$ (resp.\ $(\widehat{A_0})^*_A\xrightarrow{\cong} \widehat{A^\circ}$). 

\end{enumerate}
 
\item
Conversely, if there exists some $c\geq 0$ for which $t^c(\widehat{A_0})^+_{A'}\subset \widehat{A_0}$ (resp. $t^c(\widehat{A_0})^*_{A'}\subset \widehat{A_0}$), then one has $t^cA^+\subset A_0$ (resp. $t^cA^\circ\subset A_0$). 
\end{enumerate}
\end{proposition}

To prove this, let us verify a fundamental lemma.

\begin{lemma}
\label{CompCokerKilled}
Let $A_0$ be a ring and let $I_0\subset A_0$ be a finitely generated ideal. Let $f_0: N_0 \hookrightarrow M_0$ be an injective homomorphism between $A_0$-modules. 
Denote by $\widehat{M_0}$ and $\widehat{N_0}$ the $I_0$-adic completions of $M_0$ and $N_0$, respectively. Let $\widehat{f_0}: \widehat{N_0}\to\widehat{M_0}$ be the $\widehat{A_0}$-linear map induced by $f_0$. 
Assume that the cokernel of $f_0$ is annihilated by $I_0^m$ for some $m\geq 0$. 
Then, $\widehat{f_0}$ is also injective and the cokernel of $\widehat{f_0}$ is annihilated by $I_0^m\widehat{A_0}$. 
\end{lemma}

\begin{proof}[Proof of Lemma \ref{CompCokerKilled}]
Consider the exact sequence: $0 \to N_0 \xrightarrow{f_0} M_0 \to L_0 \to 0$ with $L_0$ being the cokernel of $f_0$. 
By assumption, the $A_0$-module $L_0$ is killed by $I_0^m$. For an arbitrary $n>0$, we have the induced exact sequence:
\begin{equation}
\label{sequence1}
0 \to N_0/(I_0^n M_0 \cap N_0) \to M_0/I_0^nM_0 \to L_0/I_0^nL_0 \to 0.
\end{equation}
Since $\varprojlim^1_{n} {N_0}/(I_0^n M_0 \cap N_0) \cong 0$ and $I_0^{m+n}M_0 \cap N_0 \subset I_0^n N_0$, the following sequence induced by $(\ref{sequence1})$ is exact:
$$
0 \to \widehat{N_0} \to \widehat{M_0} \to \widehat{L_0} \cong {L_0} \to 0,
$$
where $\widehat{L_0}$ denotes the $I_0$-adic completion of $L_0$. In particular, the cokernel of $\widehat{f_0}$ is killed by $I_0^m\widehat{A_0}$. This yields the assertion. 
\end{proof}

Moreover, we need the following result from \cite{BL95}; see also \cite[Tag 0BNR]{Stacks}.

\begin{lemma}[Beauville-Laszlo]
\label{Beauville-Laszlo}
Let $A_0$ be a ring with a nonzero divisor $t\in A_0$ and let $\widehat{A_0}$ be the $t$-adic completion. Then $t$ is a nonzero divisor of $\widehat{A_0}$ and one has the commutative diagram:
\begin{equation}
\begin{CD}\label
{BLdiagramBig}
A_0@>\psi>> \widehat{A_0} \\
@V\iota VV @VV\iota'V \\
A_0[\frac{1}{t}] @>>\psi_t>\ \widehat{A_0}[\frac{1}{t}]
\end{CD}
\end{equation}that is cartesian. In other words, we have $A_0 \cong A_0[\frac{1}{t}] \times_{\widehat{A_0}[\frac{1}{t}]} \widehat{A_0}$.
\end{lemma}

\begin{corollary}\label{newcorollary2.7}
Keep the notation as in Lemma \ref{Beauville-Laszlo}. 
Suppose that $A_0$ is completely integrally closed in $A_0[\frac{1}{t}]$. Then $\widehat{A_0}$ is completely integrally closed in $\widehat{A_0}[\frac{1}{t}]$. 
\end{corollary}

\begin{proof}
Put $A':=\widehat{A_0}[\frac{1}{t}]$. Pick an arbitrary element $x\in (\widehat{A_0})^*_{A'}$. 
Then $t^dx\in \widehat{A_0}$ for some $d>0$. Hence there exists some $a\in A_0$ such that $t^dx-\psi(a)\in t^d\widehat{A_0}$ or equivalently, $x-\psi_t(\frac{a}{t^d})\in \widehat{A_0}$. 
Then, since $x$ and any element in $\widehat{A_0}$ are contained in $(\widehat{A_0})^*_{A'}$, we find that $\psi_t(\frac{a}{t^d}) \in (\widehat{A_0})^*_{A'}$. Hence, there exists some $c>0$ such that $\psi_t(t^c(\frac{a}{t^d})^n)\in \widehat{A_0}$ for every $n>0$. Thus by Lemma \ref{Beauville-Laszlo}, we have $\psi_t(t^c(\frac{a}{t^d})^n)\in \psi_t(A_0)$ for every $n>0$. Now notice that $\ker(\psi_t)\subset A_0$. This is because $\ker(\psi)=\cap_{n=0}^\infty t^nA_0$ and $\ker(\psi)$ is $t$-divisible. Consequently, we find that $\frac{a}{t^d}\in A_0[\frac{1}{t}]$ is almost integral over $A_0$ and therefore, it is contained in $A_0$ by assumption. Hence $x$ lies in $\widehat{A_0}$, as desired. 
\end{proof}

Now let us start to prove Proposition \ref{CICprop}.

\begin{proof}[Proof of Proposition \ref{CICprop}]
Notice that $\widehat{A_0}$, $\widehat{A^+}$ and $\widehat{A^\circ}$ are $t$-torsion free (cf.\ Lemma \ref{Beauville-Laszlo}). 

We first prove the assertion $(1)$. By assumption, $t^cA^+\subset A_0$ (resp.\ $t^cA^\circ\subset A_0$) for some $c\geq 0$. Hence by applying Lemma \ref{CompCokerKilled} to the inclusion map $A_0\hookrightarrow A^+$ (resp.\ $A_0\hookrightarrow A^\circ$), we find that the induced map $\widehat{A_0}\to \widehat{A^+}$ (resp.\ $\widehat{A_0} \to  \widehat{A^\circ}$) is injective and its cokernel is killed by $t^c$. 
Therefore, we have a canonical $\widehat{A_0}$-isomorphism $A'\xrightarrow{\cong}\widehat{A^+}[\frac{1}{t}]$ (resp.\ $A'\xrightarrow{\cong}\widehat{A^\circ}[\frac{1}{t}]$). Moreover, $\widehat{A^+}$ (resp.\ $\widehat{A^\circ}$) is integrally closed (resp.\ completely integrally closed) in $\widehat{A^+}[\frac{1}{t}]$ (resp.\ $\widehat{A^\circ}[\frac{1}{t}]$) by \cite[Lemma 5.1.2]{Bh17} (resp.\ Corollary \ref{newcorollary2.7}).\footnote{The complete integral closedness of $A^\circ$ in $A$ is a not trivial issue. However, this is easily checked from the hypothesis $t^c(A_0)_A^* \subset A_0$. We will discuss this condition as \textit{preuniformity} in the following section.} Hence we have inclusions $\widehat{A_0}\subset (\widehat{A_0})^+_{A'}\subset \widehat{A^+}$ (resp.\ $\widehat{A_0}\subset (\widehat{A_0})^*_{A'}\subset \widehat{A^\circ}$). Thus, we also have inclusions 
$t^c(\widehat{A_0})^+_{A'}\subset t^c\widehat{A^+}\subset \widehat{A_0}$ (resp.\ $t^c(\widehat{A_0})^*_{A'}\subset t^c\widehat{A^\circ}\subset \widehat{A_0}$), which yields the assertion $(a)$. In particular, $\{t^n\widehat{A_0}\}_{n\geq 1}$ gives a fundamental system of open neighborhoods of $0\in (\widehat{A_0})^+_{A'}$ (resp.\ $0\in (\widehat{A_0})^*_{A'}$). Hence $(\widehat{A_0})^+_{A'}$ (resp.\ $(\widehat{A_0})^*_{A'}$) is $t$-adically complete and separated. 
Thus, by the universal property of completion (cf.\ \cite[Proposition 7.1.9 in Chapter 0]{FK18}), we obtain the $A^+$-linear map (resp.\ $A^\circ$-linear map) $\widehat{A^+}\to (\widehat{A_0})^+_{A'}$ (resp.\ $\widehat{A^\circ}\to (\widehat{A_0})^*_{A'}$) and the composite $\widehat{A^+}\to (\widehat{A_0})^+_{A'}\hookrightarrow{\widehat{A^+}}$ (resp.\ $\widehat{A^\circ}\to (\widehat{A_0})^*_{A'}\hookrightarrow\widehat{A^\circ}$) is the identity. Therefore $(\widehat{A_0})^+_{A'}\hookrightarrow{\widehat{A^+}}$ (resp.\ $(\widehat{A_0})^*_{A'}\hookrightarrow{\widehat{A^\circ}}$) is an isomorphism, which yields the assertion $(b)$. 

Next we show the assertion $(2)$. We consider the commutative diagram (\ref{BLdiagramBig}). Keeping the notation as above, assume that 
$t^c(\widehat{A_0})^+_{A'}\subset \widehat{A_0}$ (resp.\ $t^c(\widehat{A_0})^*_{A'}\subset \widehat{A_0}$) for some $c\geq 0$. Pick an element $x\in A^+$ (resp.\ $y\in A^\circ$). Then one can check that $\psi_t(x)\in \widehat{A_0}[\frac{1}{t}]$ (resp.\ $\psi_t(y)\in \widehat{A_0}[\frac{1}{t}]$) is integral (resp.\ almost integral) over $\widehat{A_0}$, because the diagram  (\ref{BLdiagramBig}) commutes. Hence by assumption, $\psi_t(t^cx)$ (resp.\ $\psi_t(t^cy)$) comes from $\widehat{A_0}$. Thus, letting $B_0$ (resp.\ $C_0$) be the $A_0$-subalgebra of $A_0[\frac{1}{t}]$ generated by all elements of $t^cA^+$ (resp.\ $t^cA^\circ$), we find that the composite map $B_0\hookrightarrow A_0[\frac{1}{t}]\xrightarrow{\psi_t}\widehat{A_0}[\frac{1}{t}]$ (resp.\ $C_0\hookrightarrow A_0[\frac{1}{t}]\xrightarrow{\psi_t}\widehat{A_0}[\frac{1}{t}]$) factors through $\widehat{A_0}$. Therefore, Lemma \ref{Beauville-Laszlo} implies that $B_0\subset A_0$ (resp.\ $C_0\subset A_0$). Consequently, we have $t^cA^+\subset A_0$ (resp.\ $t^cA^\circ\subset A$), as wanted. 
\end{proof}

\begin{corollary}
\label{Bhlem}
Keep the notation as in Proposition \ref{CICprop}. Then $(A_0)^+_A=A_0$ (resp. $(A_0)^*_A=A_0$) 
if and only if 
$(\widehat{A_0})^+_{A'}=\widehat{A_0}$ (resp.\ $(\widehat{A_0})^*_{A'}=\widehat{A_0}$).
\end{corollary}

\subsection{Tate rings}
\label{nt21}

We first recall basic terms on Tate rings.

\begin{definition}[Boundedness]
\label{defbounded}
Let $A$ be a topological ring. We say that a subset $S\subset A$ is \emph{bounded}, if for every open neighborhood $U$ of $0\in A$ there exists some open neighborhood $V$ of $0\in A$ such that $V\cdot S\subset U$, and we consider the empty set as being bounded.
\end{definition}

\begin{definition}[Tate ring]
\label{TateDefinition}
A topological ring $A$ is called \textit{Tate}, if there is an open subring $A_0 \subset A$ together with an element $t \in A_0$ such that the topology on $A_0$ induced from $A$ is $t$-adic and $t$ becomes a unit in $A$. $A_0$ is called a \textit{ring of definition} and $t$ is called a \textit{pseudouniformizer} and the pair $(A_0, (t))$ is called a \textit{pair of definition}. Denote by $A^\circ$ the set of powerbounded elements of $A$ and by $A^{\circ\circ}$ the set of all topologically nilpotent elements of $A$. Then $A^{\circ\circ}$ is an ideal of $A^\circ$ and $A^\circ$ is a subring of $A$.
\end{definition}

Any Tate ring comes from a pair of a ring and a nonzero divisor in it as follows.

\begin{lemma}
\label{lem1282}
The following assertions hold:
\begin{enumerate}

\item
Let $A_0$ be a ring with a nonzero divisor $t$, and put $A:=A_0[\frac{1}{t}]$. Equip $A$ with the linear topology defined by $\{t^nA_0\}_{n\geq 1}$. Then $A$ is equipped with the structure as a Tate ring with a ring of definition $A_0$ and a pseudouniformizer $t \in A_0$. 

\item
Conversely, let $A$ be a Tate ring with a ring of definition $A_0$ and a pseudouniformizer $t\in A_0$. Then one has $A=A_0[\frac{1}{t}]$. 

\end{enumerate}
\end{lemma}

\begin{proof}
$(1)$ is easy to check. Let us prove $(2)$. Pick an element $a\in A$. It suffices to find some integer $c>0$ such that $t^{c}a\in A_{0}$. Now by definition, the multiplication map 
$$
m: A\times A\to A,\ (a_{1}, a_{2})\mapsto a_{1}a_{2}
$$
is continuous. In particular, for every open neighborhood $V$ of $0\in A$, there exists some open neighborhood $U$ of $(a,0)\in A\times A$  such that $m(U)\subset V$. Thus, since $\{t^{n}A_{0}\}_{n\geq 1}$ forms a fundamental system of open neighborhoods of $0\in A$ by assumption, we have $m(\{a\}\times t^{c}A_{0})\subset A_{0}$ for some $c>0$. Hence the assertion follows.  
\end{proof}

\begin{definition}
\label{def211914}
Let $(A_0, I_0)$ be a pair. If  $I_0$ is generated by a nonzero divisor $t\in A_0$, then we call the Tate ring $A_0[\frac{1}{t}]$ in Lemma \ref{lem1282}(1) the \emph{Tate ring associated to $(A_0, I_0)$. }
\end{definition}

The notion of integrality is useful for describing important subrings of a Tate ring. The following lemma should be well-known, but we insert its proof.

\begin{lemma}
\label{lemma10092}
Let $A$ be a Tate ring with a ring of definition $A_0$ and a pseudouniformizer $t\in A_0$. 
\begin{enumerate}
\item
The complete integral closure of $A_0$ in $A$ coincides with $A^{\circ}$.  
In particular, if $A^{\circ}$ is bounded, then $A^{\circ}$ is completely integrally closed in $A$. 

\item
One has $tA^\circ\subset (A_0)^+_A\subset A^\circ$. 
\end{enumerate}
\end{lemma}

\begin{proof}
Let us prove $(1)$. For an element $a\in A$, $a$ is almost integral over $A_{0}$ if and only if there exists some $c>0$ such that $t^{c}a^{m}\in A_{0}$ for every $m>0$. Here the latter condition is equivalent to the condition that $a$ belongs to $A^{\circ}$. Hence $A^{\circ}$ is the complete integral closure of $A_{0}$ in $A$. The second statement is clear, because any open and bounded subring of $A$ forms a ring of definition. Next we prove $(2)$. Pick $a\in A^\circ$. Then one has $(ta)^c=t^ca^c\in A_0$ for some $c>0$. Hence $ta\in (A_0)^+_A$, as desired.
\end{proof}

\subsection{Preuniform rings and pairs}
We will discuss \emph{(pre)uniformity} of Tate rings in many contexts later. Here we give the definition.

\begin{definition}[Preuniform rings]
\label{def204}
Let $A$ be a Tate ring. We say that $A$ is \emph{preuniform} if $A^\circ$ is bounded. We say that $A$ is \emph{uniform} if it is preuniform and complete and separated. 
\end{definition}

Let us recall the following fact.

\begin{lemma}
\label{SepPUniRed}
A separated preuniform Tate ring is reduced. In particular, a uniform Tate ring is reduced.
\end{lemma}

\begin{proof}
Let $A$ be a separated preuniform Tate ring. Then, since $A^\circ$ is bounded, we can take $A^\circ$ as a ring of definition of $A$. 
Pick a pseudouniformizer $t\in A^\circ$ of $A$. Then $A^\circ$ is $t$-adically separated because $A$ is separated. 
Let $f \in A$ be such that $f^k=0$ for some $k>0$. Then for an integer $n>0$, $t^{-n}f$ is nilpotent, and therefore we get $t^{-n}f \in A^\circ$. Hence $f \in t^nA^\circ$. Since $A^\circ$ is $t$-adically separated and $n$ is arbitrary, it follows that $f=0$.
\end{proof}

We define preuniformity also for pairs that induce Tate rings.

\begin{definition}[Preuniform pairs]
\label{defpair}
Let $(A_0, I_0)$ be a pair of a ring $A_0$ and an ideal $I_0\subset A_0$. 
\begin{itemize}
\item[(1)]We say that $(A_0, I_0)$ is \emph{preuniform}, if $A_0$ has a nonzero divisor $t$ with the following property: 
\begin{itemize}
\item[$\bullet$]$I_0=tA_0$, and there exists some $c>0$ for which $t^c(A_0)^+_{A_0[\frac{1}{t}]}\subset A_0$. 
\end{itemize}
\item[(2)]We say that $(A_0, I_0)$ is \emph{uniform}, if it is preuniform and $A_0$ is $I_0$-adically complete and separated.  
\end{itemize}
\end{definition}

Recall that a morphism of pairs $f_0:(A_0,I_0) \to (B_0,J_0)$ is required to be \textit{continuous}. In other words, there exists $n>0$ such that $I_0^n \subset f_0^{-1}(J_0)$. The above definition is derived from the following fact.

\begin{lemma}
\label{lem1301423}
Let $A_0$ be a ring with a nonzero divisor $t\in A_0$ and let $A$ be the Tate ring associated to $(A_0, (t))$. 
Then $A$ is preuniform (resp.\ uniform) if and only if the pair $(A_0, (t))$ is preuniform (resp.\ uniform). 
\end{lemma}

\begin{proof}
It follows immediately from Lemma \ref{lemma10092}(2). 
\end{proof}

Let us give some guiding examples to understand the notion of preuniformity.

\begin{example}\label{EgPreUni}
\begin{enumerate}
\item
Let $V$ be a valuation ring, and $t\in V$ a nonzero element. Then, since $V$ is integrally closed in the  field of fractions $K:=\textnormal{Frac}(V)$, we have $V^{+}_{V[\frac{1}{t}]}=V$. 
Hence the pair $(V, (t))$ and the associated Tate ring $V[\frac{1}{t}]$ are preuniform. 
If $V$ is $t$-adically separated, then the associated Tate ring $V[\frac{1}{t}]$ coincides with 
$K$ as a ring (when $V\neq K$, the converse also holds; cf.\ \cite[Proposition 6.7.2 in Chapter 0]{FK18}). Notice that any valuation ring of finite rank that is not a field has such a nonzero element $t$. Indeed, if $V$ is so, then it contains a prime ideal $\mathfrak{p}$ of height $1$. Since the localization $V_{\mathfrak{p}}$ is a valuation ring of rank $1$ with maximal ideal $\mathfrak{p}V_{\mathfrak{p}}$, any nonzero element $t\in\mathfrak{p}$ satisfies $V_{\mathfrak{p}}[\frac{1}{t}]=K$, which implies  $V[\frac{1}{t}]=K$ because 
$tV_{\mathfrak{p}}\subset\mathfrak{p}V_{\mathfrak{p}}\subset V$. 
 
\item
Let $A_{0}$ be the ring of dual numbers $\mathbb{Z}_{p}[T]/(T^{2})$ over $\mathbb{Z}_{p}$. 
Then $p\in A_{0}$ is a nonzero divisor. Moreover, for every $n>0$, $\frac{T}{p^{n}}\in A_{0}[\frac{1}{p}]$ is integral over $A_{0}$ because $(\frac{T}{p^{n}})^{2}=0$. However, $p^{n-1}(\frac{T}{p^{n}})\notin A_{0}$. Hence the pair $(A_{0}, (p))$ is \emph{not} preuniform. Let $A$ be the Tate ring associated to $(A_{0}, (p))$. Then $A$ is not preuniform but separated, and it would give a counter-example to Lemma \ref{SepPUniRed} if the assumption of preuniformity is dropped. 
\end{enumerate}
\end{example}

\subsubsection{Finitely generated modules over a Tate ring}
One can define a canonical topology on a finitely generated module over a Tate ring. We introduce a notation. Let $A$ be a ring with an ideal $I \subset$ and let $f:N \to M$ be a homomorphism of $A$-modules. Set
\begin{equation}
\label{colonideal}
(N:_{M} J):=\{m \in M~|~Jm \subset f(N)\},
\end{equation}
which is a $A$-submodule of $M$.

\begin{lemma}
\label{topfinmod}
Let $A$ be a Tate ring and let $M$ be a finitely generated $A$-module. Take a ring of definition $A_0\subset A$, a pseudouniformizer $t\in A_0$ and a finite generating set $S$ of $M$ over $A$. Let $M_0\subset M$ be the $A_0$-submodule generated by $S$. Equip $M$ with the linear topology defined by $\{t^nM_0\}_{n>0}$. 
\begin{enumerate}
\item
The topology on $M$ is independent of the choices of $A_0$, $t$ and $S$. 

\item
For every finitely generated $A_0$-submodule $N_0$ of $M$ such that $M=N_0[\frac{1}{t}]$, the induced topology on $N_0$ coincides with the $t$-adic topology. 

\item
Let $f: M\to N$ be a homomorphism of $A$-modules, where $N$ is finitely generated. Equip $N$ with the topology defined above. Then $f$ is continuous. 
\end{enumerate}
\end{lemma}

\begin{proof}
Since $M=\bigcup_{n>0}(M_0:_{M}(t^n))$, there exists some $c>0$ such that  $t^cN_0\subset M_0$ for every finitely generated $A_0$-submodule $N_0$ of $M$. Hence $(2)$ and $(3)$ are easy to see. To show $(1)$, let us consider another data: $(A'_0, t', S', M'_0)$. Pick an integer $m>0$. Then it suffices to check that there exists some $m'>0$ for which $t'^{m'}M'_0\subset t^mM_0$ holds. Let $N'_0\subset M$ be the $A'_0$-submodule generated by $S$. Then $t'^{c_1}M'_0\subset N'_0$ for some $c_1>0$, as $M'_0$ is finitely generated. Meanwhile, since $t^mA_0\subset A$ is open, there exists some $c_2>0$ such that $t'^{c_2}A'_0\subset t^mA_0$ and so $t'^{c_2}N'_0\subset t^mM_0$. Hence by putting $m':=c_1+c_2$, we obtain $t'^{m'}M'_0\subset t^mM_0$, as wanted. 
\end{proof}

Notice that one can set $M=A$ in Proposition \ref{topfinmod}, and the resulting topology on $A$ coincides with the original one. Now we can give a canonical Tate ring structure to any module-finite algebra extension of a Tate ring.

\begin{lemma}
\label{tatefinex}
Let $A$ be a Tate ring and let $B$ be a module-finite $A$-algebra. Equip $B$ with the topology as in Lemma \ref{topfinmod}. Then $B$ is equipped with the structure as a Tate ring with the following property:  
\begin{itemize}
\item{for every ring of definition $A_0$ and every pseudouniformizer $t\in A_0$ of $A$, there exists a ring of definition $B_0$ of $B$  that is an integral $A_0$-subalgebra of $B$ with finitely many generators and $t\in B_0$ is a pseudouniformizer of $B$. }
\end{itemize}
\end{lemma}

\begin{proof}
Take a system of generators $x_1,\ldots, x_r$ of the $A$-module $B$. Multiplying each $x_i$ by a power of $t$ if necessary, we may assume that they are integral over $A_0$. Let $B_0\subset B$ be an $A_0$-subalgebra generated by $x_1,\ldots, x_r$. As $B=B_0[\frac{1}{t}]$, we can introduce a Tate ring structure into $B$ with a ring of definition $B_0$ and a pseudouniformizer $t\in B_0$ as in Lemma \ref{lem1282}(1). Meanwhile, since each $x_i$ is integral over $A_0$, $B_0$ is a module-finite $A_0$-algebra. Hence the topology on $B$ coincides with the one defined by setting $M=B$ in Lemma \ref{topfinmod}.
\end{proof}

If a finitely generated module $M$ over a Tate ring admits a structure of a finitely generated module over another Tate ring, then one can consider two canonical topologies on $M$. In the following situation, these topologies coincide.

\begin{lemma}
Let $A$ be a Tate ring and let $B$ be a module-finite $A$-algebra. Equip $B$ with the canonical structure as a Tate ring as in Lemma \ref{tatefinex}. 
Let $M$ be a finitely generated $B$-module. Then the following two topologies: 
\begin{itemize}
\item
the canonical topology on $M$ as a finitely generated $A$-module;
\item
the canonical topology on $M$ as a finitely generated $B$-module;
\end{itemize}
coincide. 
\end{lemma}

\begin{proof}
Let $A_0$ be a ring of definition of $A$ and let $t\in A_0$ be a pseudouniformizer of $A$. Then we can take a ring of definition $B_0$ of $B$ that is finitely generated over $A_0$ and satisfies $B=B_0[\frac{1}{t}]$. 
Let $M_0$ be a finitely generated $B_0$-submodule of $M$ such that $M=M_0[\frac{1}{t}]$. 
Then, also as an $A_0$-module, $M_0$ is finitely generated and satisfies $M=M_0[\frac{1}{t}]$. Hence the assertion follows.  
\end{proof}

\begin{remark}
\label{rmk1226}
Let $A_0$ be a ring with a nonzero divisor $t\in A_0$, and put $A:=A_0[\frac{1}{t}]$. 
Let $M$ be an $A$-module and let $M_{0}\subset M$ be an $A_{0}$-submodule such that $M=M_0[\frac{1}{t}]$. Equip $M$ with the linear topology defined by $\{t^{n}M_{0}\}_{n\geq 1}$ (this situation already occurred in the above two lemmas). 
Now let us consider the completion $M'$ of $M$ (i.e.\ $M'=\varprojlim_{n}M/t^{n}{M_{0}}$). One applies the same operations to both $A_0$ and $A$. That is, $\widehat{A_0}:=\varprojlim_{n} A_0/t^n A_0$ and $A':=\widehat{A}=\varprojlim_n A/t^n A_0$. By \cite[Chaptires III, Paragraph 6.5, Proposition 6, and Chapitres II, Paragraph 3.9, Corollaire 1]{Bour71}, one checks that $\widehat{A_0}$ and $\widehat{A}$ are rings. Now since $M_{0}$ is $t$-torsion free, so is the $t$-adic completion $\widehat{M_{0}}$ (cf.\ \cite[Lemma 4.2]{Sh15}). We equip $(\widehat{M_{0}})[\frac{1}{t}]$ with the linear topology defined by $\{t^{n}\widehat{M_{0}}\}_{n\geq 1}$. Then $(\widehat{M_{0}})[\frac{1}{t}]$ is complete and separated. Since $M_{0}$ and $\widehat{M_{0}}$ are $t$-torsion free, one can check that the map
\begin{eqnarray}
\label{mtmhat}
(\widehat{M_{0}})[\frac{1}{t}] \to M',\ \frac{(m_{n}\ \textnormal{mod}\ t^{n}M_{0})_{n\geq 1}}{t^{h}}\mapsto \bigl(\frac{m_{n+h}}{t^{h}}\ \textnormal{mod}\ t^{n}M_{0}\bigr)_{n\geq 1}
\end{eqnarray}
is well-defined. Indeed, consider the exact sequence $0 \to M_0 \to M \to M/M_0 \to 0$. Then it induces another exact sequence $0 \to M_0/t^nM_0 \to M/t^n M_0 \to M/M_0 \to 0$. Taking inverse limits, with respect to $n \in \mathbb{N}$, we obtain an exact sequence
\begin{eqnarray}
\label{exactseq}
0 \to \widehat{M_0} \to M' \to M/M_0 \to 0.
\end{eqnarray}
Since we are assuming the condition $M=\bigcup_{n\geq 1}(M_{0}:_{M} (t^{n}))$, tensoring $A_0[\frac{1}{t}]$ with the sequence (\ref{exactseq}), it follows that (\ref{mtmhat}) is an isomorphism of $A$-modules. Next we observe that (\ref{mtmhat}) gives an isomorphism of topological $A_{0}$-modules. The map (\ref{mtmhat}) is a continuous $A_{0}$-homomorphism, and the map $M_{0}\to\widehat{M_{0}}$ induces the continuous $A_{0}$-homomorphism $M\to (\widehat{M_{0}})[\frac{1}{t}]$. Moreover, the resulting diagram
\[
  \xymatrix{
(\widehat{M_{0}})[\frac{1}{t}]\ar[rr]&&M'\\
&\ar[ul]\ar[ur]M&
}
\]
commutes. Hence it follows from the universality of completion (cf.\ \cite[Proposition 7.1.9 in Chapter 0]{FK18}) that the map (\ref{mtmhat}) is an $A_{0}$-isomorphism that is also a homeomorphism. 
\end{remark}

Next let us consider a base extension of a module-finite algebra over a Tate ring. Then one may define two types of canonical topologies 
on it, but they are the same.

\begin{lemma}
Let $A$ be a Tate ring, let $(A_0, (t))$ be a pair of definition of $A$ and let $A \to A'$ be a continuous ring map between Tate rings. 
Let $B$ be a module-finite $A$-algebra, and set $B':=B\otimes_AA'$.  Equip $B$ (resp.\ $B'$) with the canonical topology by regarding it as a module-finite $A$-algebra (resp.\ $A'$-algebra). Let $B_0$ be a ring of definition of $B$ that is an $A_0$-subalgebra of $B$, and let $A'_0$ be a ring of definition of $A'$ that is an $A_0$-subalgebra of $A'$. Let $B'_0$ be the image of the natural map $B_0\otimes_{A_0}A'_0\to B'$. Then $(B'_0, (t))$ is a pair of definition of $B'$. 
\end{lemma}

\begin{proof}
Since $A=A_0[\frac{1}{t}]$, $A'=A'_0[\frac{1}{t}]$, and $B=B_0[\frac{1}{t}]$, we have $B'=B'_0[\frac{1}{t}]$. 
Thus by Lemma \ref{lem1282}(1), it suffices to check that $\{t^nB'_0\}_{n\geq 1}$ forms a fundamental system of neighborhoods of $0\in B'$. 
Now we may assume that $B_0$ is generated by finitely many elements $x_1,\ldots, x_s\in B_0$ as an $A_0$-submodule of $B$ (such a ring of definition exists by the proof of Lemma \ref{tatefinex}). 
Put $x'_i:=x_i\otimes_A 1_{A'}$ for $i=1,\ldots, s$. Then $B'_0$ is generated by $x'_1,\ldots, x'_s$ as an $A'_0$-submodule of $B'$. 
Thus, since $B'=B'_0[\frac{1}{t}]$, the assertion follows. 
\end{proof}

\subsection{Non-archimedean seminorms}
Here we give a brief review and some new notation on non-archimedean seminorms. Our basic references are \cite[Chapter 1]{BGR84}, \cite[Chapter 2, Appendix C]{FK18} and \cite[Chapter 2]{KL15}. 
\begin{definition}
\normalfont
Let $A$ be a commutative ring. 
\begin{enumerate}
\item
A function $||\cdot||: A\to \mathbb{R}_{\geq 0}$ is called a (\emph{non-archimedean}) \emph{seminorm} on $A$, if it satisfies the following conditions.
\begin{itemize}
\item[(a)]{$||0||=0$. }
\item[(b)]{$||f-g||\leq \max\{||f||, ||g||\}$ for every $f, g\in A$. }
\item[(c)]{$||1||\leq 1$. }
\item[(d)]{$||fg||\leq ||f||||g||$ for every $f, g\in A$. }
\end{itemize}
A seminorm $||\cdot||$ on $A$ is called a \emph{norm}, if it satisfies the following condition. 
\begin{itemize}
\item[(a')]{For $f\in A$, one has $||f||=0$ if and only if $f=0$. }
\end{itemize}
\item
Let $||\cdot||$ and $||\cdot||'$ be seminorms on $A$. We say that $||\cdot||$ and $||\cdot||'$ are \emph{equivalent} (or $||\cdot||$ is \emph{equivalent} to $||\cdot||'$), if there exist real numbers $C, C'>0$ such that 
one has 
$$
||f||'\leq C||f|| \leq C'||f||' 
$$
for every $f\in A$. We mean by $||\cdot||\sim||\cdot||'$ that $||\cdot||$ is equivalent to $||\cdot||'$. 
\item
Let $||\cdot||$ be a seminorm on $A$. We say that $f\in A$ is \emph{powermultiplicative} with respect to $||\cdot||$, if $||f^n||=||f||^n$ holds for every $n>0$. We say that $f\in A$ is \emph{multiplicative} with respect to $||\cdot||$, if $||fg||=||f||||g||$ holds for every $g\in A$. We say that $||\cdot||$ is \emph{powermultiplicative} (resp.\ \emph{multiplicative}), if any $f\in A$ is powermultiplicative (resp.\ multiplicative) with respect to $||\cdot||$. 
\end{enumerate}
\end{definition}

Here we list some basic facts on seminorms. 

\begin{lemma}\label{SeminormBasic}
Let $A$ be a ring and let $||\cdot||$, $||\cdot||'$, and $||\cdot||''$ be seminorms on $A$. 
\begin{enumerate}
\item
If $||\cdot||\sim||\cdot||'$ and $||\cdot||'\sim||\cdot||''$, then $||\cdot||\sim||\cdot||''$. 
\item
If $||\cdot||$ and $||\cdot||'$ are powermultiplicative and $||\cdot||\sim||\cdot||'$, then $||\cdot||=||\cdot||'$. 
\item
Let $u$ be a unit in $A$. 
Suppose that $||\cdot||$ is not identically zero. Then $||u||\neq 0$. Moreover, $u$ is multiplicative with respect to $||\cdot||$ if and only if $||u^{-1}||=||u||^{-1}$.  
\end{enumerate}
\end{lemma}

\begin{proof}
The assertions $(1)$ and $(2)$ are easy to check. 
$(3)$ follows from {\cite[$\S$1.2.2, Proposition 4]{BGR84}}. 
\end{proof}

Let $A$ be a ring, and let $||\cdot||$ be a seminorm on $A$. We put $F_r:=\{f\in A\ |\ ||f|| \le r\}$ for every positive real number $r$. Then $F_1$ forms a subring of $A$, and each $F_r$ forms an $F_1$-submodule of $A$. Hence one can define the linear topology on $A$ such that the filtration $\{F_r\}_{r>0}$ forms a fundamental system of open neighborhoods of $0\in A$. Moreover, the ring $A$ equipped with this topology is a topological ring.  For this topological ring $A$, a subset $S\subset A$ is bounded (Definition \ref{defbounded}) if and only if $S\subset F_r$ for some $r>0$. If two seminorms $||\cdot||$ and $||\cdot||'$ on $A$ are equivalent, then they define the same topology on $A$ (but the converse does not necessarily hold). 
The following class of seminorms is useful for characterization of Tate rings.

\begin{definition}
\label{def32114}
Let $A_0$ be a ring with a nonzero divisor $t\in A_0$. We say that a seminorm $||\cdot||$ on $A_0[\frac{1}{t}]$ is \emph{associated to $(A_0, (t))$}, if there exists a real number $c>1$ such that $||\cdot||$ is equivalent to the seminorm $||\cdot||_{A_0, (t),c}$ defined by
$$
||f||_{A_0, (t), c}:=
\begin{cases}
\hfill c^{\textnormal{min}\{m\in \mathbb{Z}~|~t^{m}f\in A_0\}} \hfill &(f \notin \bigcap_{n=0}^\infty t^nA_0)
\\
\hfill 0 \hfill&(f \in \bigcap_{n=0}^\infty t^nA_0).
\end{cases}
$$
\end{definition}

\begin{lemma}
\label{Tate-vs-Normed}
Let $A$ be a topological ring, let $t\in A$ be an element and let $A_0$ be a subring of $A$. Then the following conditions are equivalent. 
\begin{itemize}
\item[$(a)$]
$A$ is a Tate ring with a ring of definition $A_0$ and a pseudouniformizer $t\in A_0$. 
\item[$(b)$]
$t\in A$ is a unit, and there exists some seminorm $||\cdot||$ on $A$ with the following properties: 
\begin{itemize}
\item[$\bullet$]
the topology on $A$ is induced by $||\cdot||$; 
\item[$\bullet$]
$||t||<1$, and $t$ is multiplicative with respect to $||\cdot||$; 
\item[$\bullet$]
$A_0=\{f\in A\ |\ ||f||\leq 1\}$. 
\end{itemize}
\end{itemize}
\end{lemma}

\begin{proof}
To see $(a)\Rightarrow (b)$, it is enough to consider the seminorm $||\cdot||_{A_0, (t), 2}$. Conversely, if $(b)$ is satisfied, then $t^nA_0=\{f\in A\ |\ ||f||\leq ||t||^n\}$ for every $n>0$ by Lemma \ref{SeminormBasic}(3). Hence one can easily deduce $(a)$ from $(b)$. 
\end{proof}

For a ring $A$ and a seminorm $||\cdot||$ on $A$, we define a function 
$||\cdot||_\textnormal{sp}: A\to \mathbb{R}_{\geq 0}$ by 
$$
||f||_\textnormal{sp}:= \inf_{n\geq 1}||f^n||^{\frac{1}{n}}\ \ (f\in A).
$$
$||\cdot||_\textnormal{sp}$ has the following properties.

\begin{lemma}[{\cite[$\S$1.3.2]{BGR84}}]
\label{LemSpectral}
Let $A$ be a ring and let $||\cdot||$ be a seminorm on $A$. 
\begin{enumerate}
\item 
One has $||f||_\textnormal{sp}=\lim_{n\to\infty}||f^n||^\frac{1}{n}$ and $||f||_{\textnormal{sp}} \le ||f||$ for every $f\in A$ and $||\cdot||_\textnormal{sp}$ is a powermultiplicative seminorm on $A$.
\item
If $f\in A$ is multiplicative with respect to $||\cdot||$, then $f$ is also multiplicative with respect to $||\cdot||_\textnormal{sp}$. 
\end{enumerate}
\end{lemma}

We call $||\cdot||_\textnormal{sp}$ the \emph{spectral seminorm} associated to $||\cdot||$. Using spectral seminorms, we obtain the following characterization of preuniformity.

\begin{lemma}
\label{preunipm}
Let $A_0$ be a ring with a nonzero divisor $t\in A_0$. 

\begin{enumerate}
\item
$t\in A_0$ is multiplicative with respect to the seminorm $||\cdot||_{A_0, (t), c}$ for every $c>1$. 
\item
The following conditions are equivalent. 
\begin{itemize}
\item[$(a)$]
$(A_0, (t))$ is preuniform. 
\item[$(b)$]
For any seminorm $||\cdot||$ associated to $(A_0, (t))$, one has $||\cdot||\sim||\cdot||_\textnormal{sp}$ (or equivalently, $||\cdot||$ is equivalent to a powermultiplicative seminorm). 
\end{itemize}
\end{enumerate}
\end{lemma}

\begin{proof}
The assertion $(1)$ is clear. Let us prove $(2)$. To see $(b)\Rightarrow (a)$, it is enough to consider the topology induced by a powermultiplicative seminorm. Here we show $(a)\Rightarrow (b)$. Since $A^\circ$ is bounded with respect to (the topology induced by) $||\cdot||$, we may assume that  
$||\cdot||=||\cdot||_{A^\circ, (t), 2}$. Then $t$ is multiplicative with respect to $||\cdot||$ and $||\cdot||_\textnormal{sp}$. 
Thus, it suffices to show the existence of some constants $C, C'>0$ such that $C||a||\leq||a||_\textnormal{sp}\leq C'||a||$ for an arbitrary $a\in A^\circ\setminus tA^\circ$. Now $a^n\notin t^nA^\circ$ holds for an arbitrary $n>0$; otherwise, $(t^{-1}a)^l$ would belong to $A^\circ$ for some $l>0$, which implies that $a\in tA^\circ$ as $A^\circ$ is integrally closed in $A_0[\frac{1}{t}]$. Hence $||t||<||a^n||^\frac{1}{n}$ and therefore, $||t||\leq ||a||_\textnormal{sp}\leq 1$. Thus we have $2^{-1}||a||\leq||a||_\textnormal{sp}\leq ||a||$, as wanted. 
\end{proof}

Moreover, if $A_0$ is a valuation ring $V$ and $V[\frac{1}{t}]$ is a field, then $||\cdot||_\textnormal{sp}$ is a multiplicative norm. 

\begin{lemma}
\label{valuation-norm}
Let $V$ be a valuation ring and assume that there is a nonzero element $t\in V$ for which $V$ is $t$-adically separated. Let $||\cdot||: V[\frac{1}{t}]\to \mathbb{R}_{\geq 0}$ be a seminorm associated to $(V, (t))$. Then the spectral seminorm $||\cdot||_\textnormal{sp}$ is multiplicative. 
\end{lemma}

\begin{proof}
Since $V$ is $t$-adically separated, it follows that $K:=V[\frac{1}{t}]$ is the field of fractions of $V$ by \cite[Proposition 6.7.2 in Chapter 0]{FK18}. In view of Lemma \ref{SeminormBasic}(1), Lemma \ref{SeminormBasic}(2) and Lemma \ref{preunipm}(2), 
we may assume that $||\cdot||=||\cdot||_{V, (t), c}$ for a real number $c>1$. It suffices to prove that any $f\in K$ is multiplicative with respect to $||\cdot||_\textnormal{sp}$. As clearly $0\in K$ is  multiplicative, we assume that $f\neq 0$. By Lemma \ref{SeminormBasic}(3), we are reduced to showing that $||f^{-1}||_\textnormal{sp}=||f||_\textnormal{sp}^{-1}$. Pick an arbitrary $g\in K^\times$ and put $\lambda(g):=\min\{m\in \mathbb{Z}\ |\ t^mg\in V\}$. 
Since $V$ is a valuation ring, we have $t^{-(\lambda(g)-1)}V\subset gV\subset t^{-\lambda(g)}V$. It implies that $t^{\lambda(g)}V\subset g^{-1}V\subset t^{\lambda(g)-1}V$. 
Therefore, we find that
$$
||g||^{-1}=c^{-\lambda(g)}\leq||g^{-1}||\leq c^{-(\lambda(g)-1)}=c||g||^{-1}\ .
$$
Thus we have $(||f^n||^{\frac{1}{n}})^{-1}\leq ||f^{-n}||^\frac{1}{n}\leq c^\frac{1}{n}(||f^n||^\frac{1}{n})^{-1}$ for every $n>0$. Taking the limits, we obtain $||f||_\textnormal{sp}^{-1}=||f^{-1}||_\textnormal{sp}$, as wanted. 
\end{proof}

Now let us recall \emph{Banach rings}.

\begin{definition}[Banach rings]
\label{defbanach}\ 
\begin{enumerate}
\item
A \emph{Banach ring} is a ring $R$ equipped with a norm $||\cdot||$ such that $R$ is complete with respect to the topology defined by $||\cdot||$. 
\item
Let $R$ be a Banach ring equipped with a norm $||\cdot||$. We say that $R$ is \emph{uniform}, if $R$ contains a unit $t$ with $||t||<1$, and $||\cdot||$ is equivalent to a powermultiplicative norm. 
\item
Let $R$ and $S$ be Banach rings and denote by $||\cdot||_R$ and $||\cdot||_S$ the norms on $R$ and $S$, respectively. We say that a ring homomorphism $\varphi: R\to S$ is \emph{bounded}, if one has $||\varphi(f)||_S\leq||f||_R$ for every $f\in R$ (notice that this property is stable under composition). 

\item
Let $R$ be a Banach ring. A \emph{Banach $R$-algebra} is a Banach ring $S$ equipped with a bounded homomorphism $\varphi: R\to S$. 

\item
Let $A_0$ be a ring with a nonzero divisor $t\in A_0$ that is $t$-adically complete and separated. We say that a Banach ring $R$ is \emph{associated to $(A_0, (t))$}, if the underlying ring $R$ is equal to $A_0[\frac{1}{t}]$ and the norm on $R$ is associated to $(A_0, (t))$.  
\end{enumerate}
\end{definition}

From now on, we view a Banach ring also as a (complete and separated) topological ring by considering the topology defined by the norm. Then we can compare Banach rings with Tate rings as follows. 
\begin{lemma}
Let $R$ be a Banach ring equipped with a norm $||\cdot||$. Then the following assertions hold. 
\begin{enumerate}
\item
If $R$ contains a unit $t$ with $||t||<1$, then $R$ is a complete and separated Tate ring that has a ring of definition 
$$
R_{0}=\{f\in R\ |\ ||f||\leq 1\}
$$ 
and a pseudouniformizer $t$. 
\item
If $R$ is a Banach ring associated to $(A_0, (t))$ as in Definition \ref{defbanach}(5), then $R$ is the Tate ring associated to $(A_0, (t))$.  
\item
If $R$ is uniform (in the sense of Definition \ref{defbanach}(2)), then $R$ is a uniform Tate ring (in the sense of Definition \ref{def204}). 
\end{enumerate}
\end{lemma}
\begin{proof}
$(1)$: By definition, $t$ lies in $R_{0}$. Hence it suffices to show that $\{t^{n}R_{0}\}_{n\geq 1}$ forms a fundamental system of open neighborhoods of $0\in R$. Recall that $R$ admits a fundamental system of open neighborhoods $\{F_r\}_{r>0}$ (where $F_r=\{f\in R\ |\ ||f|| \le r\}$) at $0\in R$. Since $t^{n}R_{0}\subset F_{||t||^{n}}$ for every $n\geq 1$, we are reduced to showing that there exists some $C_{n}>0$ with $F_{C_{n}}\subset t^{n}R_{0}$.

First we consider the case when $||t^{-1}||=0$. Then for every $f\in R$, we have 
$||f||=0$ by Lemma \ref{SeminormBasic}(3), and hence $f=0$. Therefore, $R$ is the zero ring.\footnote{The zero ring $\{0\}$ equipped with the unique topology is a Tate ring whose pair of definition is $(\{0\}, (0))$. Notice that any ring is $0$-adically complete and separated. } Thus $F_{1}=t^{n}R_{0}$ for every $n\geq 1$, which yields the assertion. 

Next we assume that $||t^{-1}||\neq 0$. Set $C:=||t^{-1}||$. Let us show that $F_{C^{-n}}\subset t^{n}R_{0}$ for every $n\geq 1$. Pick $f\in R$ with $||f||\leq C^{-n}$. 
Then we have 
$$
||t^{-n}f||\leq ||t^{-1}||^{n}||f||= C^{n}||f||\leq 1. 
$$
Hence $t^{-n}f\in R_{0}$, which means that $f\in t^{n}R_{0}$. Thus the assertion follows. 

$(2)$: It follows immediately from Lemma \ref{Tate-vs-Normed}. 

$(3)$: Let $||\cdot||'$ be a powermultiplicative norm on $R$ that is equivalent to $||\cdot||$. Then $||t^{l}||'<1$ for some $l\geq 1$, and hence $||t||'<1$ because $||\cdot||'$ is powermultiplicative. Thus, since $||\cdot||$ and $||\cdot||'$ define the same topology on $R$, $R$ is a Tate ring that has a ring of definition 
$$
R'_{0}:=\{f\in R\ |\ ||f||'\leq 1\}
$$
and a pseudouniformizer $t$ by the assertion (1). 
Moreover, since $||\cdot||'$ is powermultiplicative, $R'_{0}$ is completely integrally closed in $R=R'_{0}[\frac{1}{t}]$. Hence the assertion follows from Lemma \ref{lemma10092}.
\end{proof}

An \emph{almost perfectoid algebra} (cf.\ Definition \ref{defalmostperf}) is an important example of a Banach algebra. We use the following lemma to describe the relationship between almost perfectoid algebras and almost Witt-perfect algebras (cf.\ Proposition \ref{perfalgdecomp}).

\begin{lemma}
\label{BanachAlg}
Let $A_0$ be a ring with a nonzero divisor $t \in A_0$ and let $B_0$ be a $t$-torsion free $A_0$-algebra. Assume that $A_0$ and $B_0$ are $t$-adically complete and separated. 
Then there exist Banach rings $R$ and $S$ with the following properties: 
\begin{itemize}
\item
$R$ is associated to $(A_0, (t))$ and $S$ is associated to $(B_0, (t))$;
\item
the ring homomorphism $R\to S$ induced by the homomorphism $A_0\to B_0$ is bounded. 
\end{itemize}
If further $(A_0, (t))$ and $(B_0, (t))$ are uniform, then one can take $R$ so that the norm on $R$ is powermultiplicative. 
\end{lemma}

\begin{proof}
To see the first assertion, it suffices to take a real number $c>1$ and equip $A_0[\frac{1}{t}]$ and $B_0[\frac{1}{t}]$ with the norms $||\cdot||_{A_0, (t), c}$ and $||\cdot||_{B_0, (t), c}$, respectively. To check the last assertion, it is enough to replace $||\cdot||_{A_0, (t), c}$ and $||\cdot||_{B_0, (t), c}$ with their respective spectral norms.
\end{proof}

Hereafter, we use the following notation: for a Banach ring $R$ with a norm $||\cdot||$, we equip the ring $R[T^\frac{1}{p^\infty}]$ with the norm $||\cdot||_\textnormal{Gauss}$ defined by $||\sum_{h\in \mathbb{Z}[\frac{1}{p}]}r_hT^h||_\textnormal{Gauss}:=\sup_{h\in \mathbb{Z}[\frac{1}{p}]}\{||r_h||\}$ (where $r_h\in R$), and denote by $R\langle T^\frac{1}{p^\infty}\rangle$ the resulting Banach ring obtained by completion.

\section{Almost ring theory and almost surjectivity of the Frobenius map}

\subsection{Basic setup and semiperfect rings}
We start with basic part of the theory of almost rings and modules. We say that a pair $(R,I)$ is a \textit{basic setup}, if $I$ is an ideal of a ring $R$, $I=I^2$ and $I$ is a flat $R$-module.\footnote{The flatness of $I$ is important for developing a solid theory of almost rings and modules.} An $R$-module map $f:M \to N$ is said to be \textit{$I$-almost injective} (resp. \textit{$I$-almost surjective}), if the kernel (resp. cokernel) of $f$ is annihilated by $I$. Moreover, $f$ is said to be an \textit{$I$-almost isomorphism}, if it is both $I$-almost injective and $I$-almost surjective. A general reference is the book \cite{GR03}. We will be mainly concerned with the following cases:
\begin{enumerate}
\item[$\bullet$]
$(R,I)=(\mathbb{Z}[T^{\frac{1}{p^\infty}}], (T)^{\frac{1}{p^\infty}})$.

\item[$\bullet$]
$(R,I)=(K^\circ\langle T^\frac{1}{p^\infty}\rangle,(T)^\frac{1}{p^\infty}K^{\circ\circ}K^\circ\langle T^\frac{1}{p^\infty}\rangle)$, where $K$ is a perfectoid field and $K^{\circ\circ}$ is the set of topologically nilpotent elements of $K$ (see Definition \ref{TateDefinition}).

\item[$\bullet$]
$(R,I)=(A_0,I)$, where $A_0$ is a ring of definition of a Tate ring $A$ (see Example \ref{BasicSetupEx} below).
\end{enumerate}

\begin{example}
\label{BasicSetupEx}
Fix a prime number $p>0$. 
Let $A_0$ be a ring with a sequence of nonzero divisors $\{t_n\}_{n\geq 0}$ such that $A_0$ is integrally closed in $A:=A_0[\frac{1}{t_0}]$ and assume that for every $n\geq 0$ we have $t_{n+1}^p=t_n u_n$ for some unit $u_n \in A_0^\times$. 
Denote by $I$ the ideal $\sqrt{(t_0)}\subset A_0$. Then the pair $(A_0, I)$ is a basic setup. Let us observe it. 
Pick $x\in I$. Then we have $x^{p^m}=t_0a$ for some $m>0$ and $a\in A_0$. Thus, since $t_m^{p^m}=t_0u$ for some unit $u \in A_0$, an equality $(\frac{x}{t_m})^{p^m}=au^{-1}$ holds in $A$. Hence $x$ lies in $t_mA_0$, because $A_0$ is integrally closed in $A$. Consequently, we have $I=\varinjlim_nt_nA_0$. Therefore, we find that $I=I^p$ and $I$ is flat over $A_0$. 
\end{example}

We give definitions of (almost) semiperfect rings that include both classes of Witt-perfect and perfectoid algebras.

\begin{definition}[Almost semiperfect ring]
Let $A_0$ be a ring and fix a prime number $p>0$.
\begin{enumerate}
\item
We say that $A_0$ is \textit{perfect} (resp. \textit{semiperfect}), if the Frobenius endomorphism on $A_0/(p)$ is bijective (resp. surjective).

\item
Assume that $A_0$ is a $\mathbb{Z}[T^{\frac{1}{p^\infty}}]$-algebra with a basic setup $(\mathbb{Z}[T^{\frac{1}{p^\infty}}], (T)^{\frac{1}{p^\infty}})$. We say that $A_0$ is \emph{$(T)^{\frac{1}{p^\infty}}$-almost semiperfect}, if the Frobenius endomorphism $F:A_0/(p) \to A_0/(p)$ is $(T)^{\frac{1}{p^\infty}}$-almost surjective, where the target ring of $F$ is regarded as an $A_0$-module via the Frobenius map.
\end{enumerate}
\end{definition}

First we give several lemmas on (almost) semiperfect rings for later use. For an element $t$ in a ring $A_0$, we say that $A_0$ is \textit{$t$-adically Zariskian}, if $t$ is contained in the Jacobson radical of $A_0$; see \cite{FK18} and \cite{T18} for details on Zariskian geometry. Notice that for any ring, one can define a trivial $\mathbb{Z}[T^\frac{1}{p^\infty}]$-algebra structure by assigning the unity to each $T^\frac{1}{p^n}$. Over such a $\mathbb{Z}[T^\frac{1}{p^\infty}]$-algebra, $(T)^{\frac{1}{p^\infty}}$-almost semiperfectness is equivalent to semiperfectness.

\begin{lemma}
\label{zarsemiperf}
Let $A_0$ be a ring with a nonzero divisor $\varpi$ such that $p\in \varpi^pA_0$. Assume that $A_0$ is $\varpi$-adically Zariskian and $A_0/(\varpi^p)$ is semiperfect. Then there exists a sequence $\{\varpi_n\}_{n\geq 0}$ in $A_0$ such that $\varpi_0=\varpi$ and for every $n \ge 0$, we have $\varpi_{n+1}^p=\varpi_n u_n$ for some unit $u_n \in A_0^\times$. 
\end{lemma}

\begin{proof}
We carry out the proof by induction. Put $\varpi_0:=\varpi$. Then, since $A_0/(\varpi_0^p)$ is semiperfect, we find $a_0, b_0\in A_0$ for which $a_0^p=\varpi_0+\varpi_0^pb_0=\varpi_0(1+\varpi_0^{p-1}b_0)$. Here $1+\varpi^{p-1}_0b_0$ is a unit in $A_0$, because $A_0$ is $\varpi$-adically Zariskian. Hence we can take $\varpi_1=a_0$. Next pick an integer $m>0$ and assume that the assertion holds true for every $n\leq m-1$. Take $a_m, b_m\in A_0$ for which $\varpi_m=a_m^p+\varpi^pb_m$. Now $(\varpi_m^{p^m})=(\varpi)$ as ideals by assumption and therefore, we have $c_m\in A$ for which $\varpi=\varpi_mc_m$. Then it holds that $\varpi_m^p=\varpi_m(1+\varpi_m^{p-1}c_m^pb_m)$ and $1+\varpi_m^{p-1}\varpi_m^pb_m \in A_0^\times$. Hence we can take $\varpi_{m+1}=a_m$, which completes the proof. 
\end{proof}

\begin{lemma}
Let $A_0$ be a $\mathbb{Z}[T^{\frac{1}{p^\infty}}]$-algebra with a basic setup $(\mathbb{Z}[T^{\frac{1}{p^\infty}}], (T)^{\frac{1}{p^\infty}})$. Then the following conditions are equivalent. 

\begin{enumerate}
\item
$A_0$ is $(T)^{\frac{1}{p^\infty}}$-almost semiperfect.
 
\item
There is a semiperfect $\mathbb{Z}[T^{\frac{1}{p^\infty}}]$-algebra $B_0$ with the following property: $A_0/(p)$ is $(T)^{\frac{1}{p^\infty}}$-almost isomorphic to $B_0/(p)$ as $\mathbb{Z}[T^{\frac{1}{p^\infty}}]$-algebras.
\end{enumerate}
\end{lemma}

\begin{proof}
First we assume $(1)$. We denote by $g^\frac{1}{p^n}$ the image of $T^\frac{1}{p^n}$ in $A_0/(p)$ for every $n\geq 0$. Consider the inverse perfection $A_0^\flat:=\varprojlim_{x\mapsto x^p}A_0/(p)$, together with a ring map $\mathbb{Z}[T^\frac{1}{p^\infty}]\to A_0^\flat$ that assigns $(g, g^\frac{1}{p}, g^\frac{1}{p^2}, \ldots)\in A_0^\flat$ to $T$. Let $\Phi_{A_0}: A_0^\flat\to A_0/(p)$ be the projection map defined by the rule $(a_0, a_1, a_2, \ldots)\mapsto a_0$. Then the induced map $A_0^\flat/\textnormal{Ker}(\Phi_{A_0})\hookrightarrow A_0/(p)$ is a $(T)^{\frac{1}{p^\infty}}$-almost isomorphism and $A_0^\flat/\textnormal{Ker}(\Phi_{A_0})$ is semiperfect. Hence the implication $(1) \Rightarrow (2)$ follows. The converse is easy to check. 
\end{proof}

For a surjective ring map $A_0\twoheadrightarrow B_0$ with $B_0$ semiperfect, clearly the semiperfectness does not lift to $A_0$ in general. On the other hand, in the situations we deal with later, the following assertion holds. 

\begin{lemma}
\label{lem21219}
Let $A_0$ be a $\mathbb{Z}[T^\frac{1}{p^\infty}]$-algebra with a nonzero divisor $\varpi$ such that $p\in\varpi^pA_0$. Assume that $\varpi$ admits a $p$-th root $\varpi^\frac{1}{p}\in A_0$. Then the following assertions hold.  
\begin{enumerate}
\item
$A_0/(\varpi)$ is $(T)^{\frac{1}{p^\infty}}$-almost semiperfect if and only if $A_0/(\varpi^p)$ is $(T)^{\frac{1}{p^\infty}}$-almost semiperfect. 
\item
Equip $A_0$ with the $\varpi$-adic topology and assume further that $(p)\subset A_0$ is closed and $A_0$ is complete and separated. 
Then $A_0/(\varpi^p)$ is $(T)^{\frac{1}{p^\infty}}$-almost semiperfect if and only if 
$A_0/(p)$ is $(T)^{\frac{1}{p^\infty}}$-almost semiperfect. 
\end{enumerate}
\end{lemma}

\begin{proof}
If {$A_0/(\varpi^p)$ (resp.\ $A_0/(p)$) is $(T)^{\frac{1}{p^\infty}}$-almost semiperfect}, then clearly so is $A_0/(\varpi)$ (resp.\ $A_0/(\varpi^p)$). 
So it suffices to prove the inverse implications. Fix an arbitrary integer $n>0$, and put $g^\frac{1}{p^{k}}:=T^\frac{1}{p^{k}}\cdot 1\in A_0$ for every $k\geq 0$. Assume that $A_0/(\varpi)$ is $(T)^{\frac{1}{p^\infty}}$-almost semiperfect. Pick an element $a\in A_0$ and put $a':=g^\frac{1}{p^n}a$. Then by assumption, there exist some $a_0,b_0 \in A_0$ such that $g^\frac{p-1}{p^{n+1}}a= a_0^p+\varpi b_0$. 
Multiplying both sides by $g^\frac{1}{p^{n+1}}$, we obtain $a'=g^\frac{1}{p^{n+1}}a_0^p+\varpi (g^\frac{1}{p^{n+1}}b_0)$. Similarly, we can find some $a_1, b_1\in A_0$ such that $g^\frac{1}{p^{n+1}}b_0=g^\frac{1}{p^{n+2}}a_1^p+\varpi(g^\frac{1}{p^{n+2}}b_1)$. This procedure yields the following assertion: if a system of elements $a_0,\ldots, a_m\in A_0$ satisfies $a'\equiv\sum^m_{i=0}g^\frac{1}{p^{n+i+1}}a_i^{p}\varpi^{i}\mod (g^\frac{1}{p^{n+m+1}}\varpi^{m+1})$, then there exists some 
$a_{m+1}\in A_0$ for which $a'\equiv\sum^{m+1}_{i=0}g^\frac{1}{p^{n+i+1}}a_i^{p}\varpi^i\mod  ( g^\frac{1}{p^{n+m+2}}\varpi^{m+2})$. Hence by axiom of choice, we obtain a sequence $\{a_m\}_{m\geq 0}$ in $A_0$ such that $a'\equiv\sum^m_{i=0}g^\frac{1}{p^{n+i+1}}a_i^{p}\varpi^{i}\mod (g^\frac{1}{p^{n+m+1}}\varpi^{m+1})$ for every $m\geq 0$. In particular, we have
$$
a'\equiv\sum^{p-1}_{i=0}g^\frac{1}{p^{n+i+1}}a_i^{p}\varpi^i\equiv\biggl(\sum^{p-1}_{i=0}g^\frac{1}{p^{n+i+2}}a_i\varpi^\frac{i}{p}\biggr)^p\mod (\varpi^{p}),  
$$
which yields $(1)$. To prove $(2)$, we equip $A_0$ with the $\varpi$-adic topology, and assume further that $(p)\subset A_0$ is closed and $A_0$ is complete and separated. Set $b_m:=\sum^m_{i=0}g^\frac{1}{p^{n+i+2}}a_i\varpi^\frac{i}{p}$ $(m\geq0)$ and $b:=\lim_{m\to\infty}b_m\in A_0$. 
Then $\sum^m_{i=0}g^\frac{1}{n+i+1}a_i^p\varpi^i-b_m^p\in (p)$ for every $m$. 
Hence it follows that 
$$
a'-b^p=\lim_{m\to\infty}\sum^{m}_{i=0}g^\frac{1}{n+i+1}a_i^p\varpi^i-\lim_{m\to \infty}b_m^p
=\lim_{m\to\infty}\Big(\sum^{m}_{i=0}g^\frac{1}{n+i+1}a_i^p\varpi^i-b_m^p\Big)\in (p),
$$
because $(p)\subset A_0$ is a closed ideal and so $(2)$ follows. 
\end{proof}

\begin{corollary}
\label{gAl-pigAl}
Let $A_0$ be a ring with a nonzero divisor $\varpi$ such that $p\in \varpi^pA_0$ and $A_0$ is integrally closed in $A_0[\frac{1}{\varpi}]$ and let $g\in A_0$ be an element. Assume that $A_0$ admits compatible systems of $p$-power roots 
$\varpi^\frac{1}{p^n}, g^\frac{1}{p^n}  \in A_0$. Then the following conditions are equivalent. 
\begin{itemize}
\item[$(a)$]
The Frobenius endomorphism on $A_0/(\varpi^p)$ is $(g)^\frac{1}{p^\infty}$-almost surjective. 
\item[$(b)$]
The Frobenius endomorphism on $A_0/(\varpi^p)$ is $(\varpi g)^\frac{1}{p^\infty}$-almost surjective. 
\end{itemize}
\end{corollary}

\begin{proof}
$(a)\Rightarrow (b)$ is clear. To show the converse, we assume $(b)$. Fix an arbitrary integer $n>0$, and put $\varpi^{-\frac{1}{p^k}}:=(\varpi^\frac{1}{p^k})^{-1}$ for every $k\geq 1$. Pick $a\in A_0$. Then there exist elements $b, c\in A_0$ such that 
$\varpi^\frac{1}{p^n}g^\frac{1}{p^n}a=b^p+\varpi^pc$ and therefore, $g^\frac{1}{p^n}a=(\varpi^{-\frac{1}{p^{n+1}}}b)^p+\varpi(\varpi^{-\frac{1}{p^n}}\varpi^{p-1}c)$ and $\varpi^{-\frac{1}{p^n}}\varpi^{p-1}c$ belongs to $A_0$. Moreover, $\varpi^{-\frac{1}{p^{n+1}}}b$ belongs to $A_0$, because 
$(\varpi^{-\frac{1}{p^{n+1}}}b)^p \in A_0$ and $A_0$ is integrally closed in $A_0[\frac{1}{\varpi}]$. Thus we see that the Frobenius endomorphism on $A_0/(\varpi)$ is $(g)^\frac{1}{p^\infty}$-almost surjective and Lemma \ref{lem21219} implies $(a)$, as wanted. 
\end{proof}

If a Tate ring $A$ is given, one can consider a subring $A^{+}\subset A^{\circ}$ that is open and integrally closed in $A$. Clearly, $A^{\circ}$ is the biggest ring having such properties. As one can observe in Example \ref{int.neq.aint}, $A^+$ and $A^\circ$ are essentially different in general.\footnote{In \cite{Sch12}, the distinction between $A^+$ and $A^\circ$ does not cause any serious issue. Lemma \ref{lem2231930} is one of the reasons.} On the other hand, they have a common feature on semiperfectness in certain situations arising from the theory of perfectoid algebras. 

\begin{lemma}
\label{lem2231930}
Let $A$ be a Tate ring. Let $A^+$ be a $\mathbb{Z}[T^\frac{1}{p^\infty}]$-algebra that is an open and integrally closed subring of $A$ contained in $A^{\circ}$. Assume that $A$ has a pseudouniformizer $\varpi$ that satisfies 
$p\in\varpi^{p}A^{+}$ and admits a $p$-th root $\varpi^{\frac{1}{p}}\in A^{+}$. 
Then the following conditions are equivalent. 
\begin{itemize}
\item[$(a)$]
$A^\circ/(\varpi^p)$ is $(T)^{\frac{1}{p^\infty}}$-almost semiperfect.

\item[$(b)$]
$A^+/(\varpi^p)$ is $(T)^{\frac{1}{p^\infty}}$-almost semiperfect. 

\end{itemize}
\end{lemma}

\begin{proof}
Fix an arbitrary integer $n>0$ and put $g^\frac{1}{p^{n}}:=T^\frac{1}{p^{n}}\cdot 1\in A^+$. 
Assume that $A^\circ/(\varpi^p)$ is $(T)^{\frac{1}{p^\infty}}$-almost semiperfect. 
Pick $a\in A^+$. Then $g^\frac{1}{p^n}a=b^p+\varpi^pc$ for some $b, c\in A^\circ$. Since $\varpi^{p-1} c\in A^+$ by Lemma \ref{lemma10092}(2), we find that $b^p\in A^+$ which gives $b\in A^+$. Hence $A^+/(\varpi)$ is $(T)^{\frac{1}{p^\infty}}$-almost semiperfect, which implies that $A^+/(\varpi^p)$ is also $(T)^{\frac{1}{p^\infty}}$-almost semiperfect in view of Lemma \ref{lem21219}. Hence $(a)\Rightarrow (b)$ holds. 
To show the converse, assume that $A^+/(\varpi^p)$ is $(T)^\frac{1}{p^\infty}$-almost semiperfect. Pick $d\in A^\circ$. Then $\varpi d \in A^+$ and thus $g^\frac{1}{p^n}\varpi d=e^p+\varpi^p f$ for some $e, f\in A^+$. 
Then $g^\frac{1}{p^n}d=(\frac{e}{\varpi^{1/p}})^p+\varpi^{p-1}f$. Since $A^\circ$ is integrally closed in $A$, we obtain $\frac{e}{\varpi^{1/p}}\in A^\circ$. Hence $A^\circ/(\varpi)$ is $(T)^\frac{1}{p^\infty}$-almost semiperfect. Therefore, $A^\circ/(\varpi^p)$ is semiperfect in view of Lemma \ref{lem21219}, as required.  
\end{proof}

\begin{proposition}
Let $A_0$ be a $p$-torsion free semiperfect ring and assume that a finite group $G$ acts on $A_0$ as ring automorphisms with $|G|$ invertible on $A_0$. Then the ring of invariants $A_0^G$ is also semiperfect.
\end{proposition}

\begin{proof}
Pick an element $x \in A_0^G$. Since $A_0$ is semiperfect, one can find elements $a,b \in A_0$ such that $a^p=x+pb$. As $|G|$ is invertible on $A_0$, the Reynolds operator: $\sigma(a):=\frac{1}{|G|} \sum_{g \in G} g(a)$ gives a well-defined element of $A_0^G$. We compute
$$
\sigma(a)^p=\Big(\frac{1}{|G|} \sum_{g \in G} g(a)\Big)^p=\Big(\frac{1}{|G|^p}\Big)
\Big(\sum_{g \in G} g(a^p)+pf\Big)~\mbox{for some}~f \in A_0.
$$
After plugging the equality $a^p=x+pb$ into the above formula, we deduce that
$$
\Big(\frac{1}{|G|^p}\Big)
\Big(\sum_{g \in G} g(a^p)+pf\Big)=\Big(\frac{1}{|G|^p}\Big)
\Big(\sum_{g \in G}(g(x)+pg(b))+pf\Big)
$$
$$
=\Big(\frac{1}{|G|^p}\Big)
(\sum_{g \in G}x)+p\Big(\frac{1}{|G|^p}\Big)
\Big(\sum_{g \in G}(g(b))+f\Big)=\frac{x}{|G|^{p-1}}+pg~\mbox{for some}~g \in A_0.
$$
We write this as
\begin{equation}
\label{Reynoldsoperator}
|G|^{p-1}\sigma(a)^p=x+pg|G|^{p-1}.
\end{equation}
For simplicity, let $m:=|G|^{p-1}$. Then considering the image of $m$ in the finite field $\mathbb{Z}/(p) \hookrightarrow A_0/(p)$, one can find integers $k,h \in \mathbb{Z}$ such that $k^p=m+ph$. After plugging this into $(\ref{Reynoldsoperator})$, we get
$$
(k\sigma(a))^p=x+pt~\mbox{for some}~t \in A_0.
$$
The Reynolds operator implies that $A_0^G \to A_0$ splits as a sequence of $A_0^G$-modules, we have $pA_0^G=A_0^G \cap pA_0$. The fact $pt=(k\sigma(a))^p-x \in A_0^G$ implies that $t \in A_0^G$. So we conclude that the image of $k\sigma(a)$ in $A_0^G/(p)$ maps to $\overline{x} \in A_0^G$ under the Frobenius map, as desired.
\end{proof}

\subsection{Witt-perfect and perfectoid algebras}

We now recall Fontaine's perfectoid rings; see \cite{Fo13} for reference.

\begin{definition}[Fontaine]
\label{Fontainedef}
We say that a Banach ring $R$ is \emph{perfectoid}, if $R$ is uniform and $R$ has a pseudouniformizer $\varpi$ 
such that $p\in \varpi^pR^\circ$ and $R^\circ/(\varpi^p)$ is semiperfect. We call such $\varpi\in R$ a \emph{perfectoid pseudouniformizer} of $R$. 
\end{definition}

Notice that for a Banach ring $R$, it only depends on the topological ring structure whether $R$ is perfectoid or not.

\begin{definition}
Let $R$ be a perfectoid Banach ring. 
\label{ScholzePerfectoid}
\begin{enumerate}
\item
A \emph{perfectoid field} is a perfectoid Banach ring $K$ that is a field and whose topology can be defined by a multiplicative norm on $K$. 
\item
A \emph{perfectoid $R$-algebra} is a Banach $R$-algebra that is perfectoid.  
\end{enumerate}
\end{definition}

Notice that our definition of a perfectoid field and a perfectoid $K$-algebra (where $K$ is a perfectoid field) coincides with Scholze's original one. To see this, it suffices to check the following.

\begin{lemma}
Let $K$ be a field equipped with a norm $||\cdot||$. 
\begin{enumerate}
\item
The following conditions are equivalent. 
\begin{itemize}
\item[$(a)$]
$K$ is a perfectoid field. 
\item[$(b)$]
$K$ is complete, $K^\circ$ is a non-Noetherian valuation ring of rank $1$, $||\cdot||$ is equivalent to an absolute value associated to $K^\circ$, $||p||<1$ and $K^\circ/(p)$ is semiperfect.
\end{itemize}
\item
Suppose that $K$ is a perfectoid field, and let $R$ be a Banach $K$-algebra. Then the following conditions are equivalent. 
\begin{itemize}
\item[$(a)$]
$R$ is a perfectoid $K$-algebra. 
\item[$(b)$]
$R$ is uniform and for every $\varpi\in K$ such that $||p||\leq||\varpi||< 1$, $R^\circ/(\varpi)$ is semiperfect. 
\end{itemize}
\end{enumerate}
\end{lemma}

\begin{proof}
$(1)$: Let us assume $(b)$ first. Then by \cite[Lemma 3.2]{Sch12}, we have some $\varpi\in K^\circ$ for which $p\in\varpi^pK^\circ$. 
Hence it is easily seen that $(a)$ holds. Next we assume $(a)$ conversely. Then $K$ is complete and $K^\circ$ is a valuation ring of rank $1$ by assumption. 
Now $||\cdot||\sim ||\cdot||_\textnormal{sp}$ by Lemma \ref{preunipm}(2), and $||\cdot||_\textnormal{sp}$ coincides with an absolute value associated to $K^\circ$ by Lemma \ref{valuation-norm}. Moreover, $||\cdot||_\textnormal{sp}$ is not discretely valued in view of \cite[Lemma 1.1(iii)]{M17}. The remaining part follows from Lemma \ref{lem21219}, because $(p)\subset K^\circ$ is closed with respect to the topology defined by $||\cdot||_\textnormal{sp}$. 

$(2)$: $(b)\Rightarrow (a)$ is clear. Now we deduce $(b)$ from $(a)$. Since $K$ is a perfectoid field, $p\in K$ is a topologically  nilpotent unit or equal to zero. The same assertion also holds for $p\in R$, as $K\to R$ is continuous. Hence $(p)\subset R^\circ$ is closed with respect to the induced topology from $R$. Thus, $(a)$ implies that $R^\circ/(p)$ is semiperfect in view of Lemma \ref{lem21219}. 
Therefore $(a)\Rightarrow (b)$ holds. 
\end{proof}

Any perfectoid pseudouniformizer can be replaced so that the following statement holds.

\begin{lemma}
\label{1151321}
Let $A$ be a perfectoid Banach ring and let $A^+$ be an open and integrally closed subring of $A$ contained in $A^\circ$. Then the following assertions hold. 
\begin{enumerate}
\item
There exists some $\varpi'\in A^+$ such that $p\in \varpi'^pA^+$ and $A^+/(\varpi'^p)$ is semiperfect. 
\item
Assume further that $p\in A$ is a topologically nilpotent unit. 
Then $A^+/(p)$ is semiperfect, and there exists some $t\in A^+$ such that $p=t^pu$ for some unit $u\in A^+$. Moreover, there exists some $s\in A^+$ such that $s^p\equiv p\mod p^2A^+$.  
\end{enumerate}
\end{lemma}

\begin{proof}
$(1)$: Let $\varpi\in A^\circ$ be a perfectoid pseudouniformizer of $A$. Then by Lemma \ref{zarsemiperf}, we find some $\varpi_1\in A^\circ$ such that $\varpi=\varpi_1^p u_1$ for some unit $u_1 \in A^\circ$ (notice that $\varpi_1$ is a topologically nilpotent unit in $A$). Hence $p\in\varpi_1^p\varpi^{p-1}A^\circ$ and since $\varpi^{p-1}A^\circ \subset A^+$, we have $p\in \varpi_1^pA^+$. Moreover, since $A^\circ/(\varpi_1^p)$ is semiperfect, $A^+/(\varpi_1^p)$ is also semiperfect by Lemma \ref{lem2231930}. So letting $\varpi':=\varpi_1$ completes $(1)$. 

$(2)$: Next assume further that $p\in A$ is a topologically nilpotent unit. Then, since $(p)\subset A^+$ is closed with respect to the $\varpi_1$-adic topology, $A^+/(p)$ is semiperfect in view of Lemma \ref{lem21219}(2). 
To prove the existence of $t\in A^+$ and a unit $u\in A^+$ such that $p=t^pu$, we take $a\in A^+$ for which $p=\varpi_1^pa$. Then we have $b, c\in A^+$ that satisfy $a=b^p+pc$ and so $p(1-\varpi_1^pc)=(\varpi_1b)^p$. Thus, since $u:=1-\varpi_1^pc\in A^+$ is a unit, it suffices to take $t=\varpi_1b$. Finally, let us prove the existence of $s\in A^+$ as in the assertion. Take $d, e\in A^+$ such that 
$u=d^p+pe$. Then $p=t^p(d^p+pe)$ and therefore, $p=(td)^p+pt^pe$. Here notice that $t^p=pu^{-1}\in pA^+$. Thus we have $(td)^p \equiv p \pmod{p^2A^+}$, as wanted. 
\end{proof}

\begin{example}
We exhibit a specific example of perfectoid ring in the sense of Fontaine, but does not fit into the original definition of perfectoid algebras by Scholze. Fix a prime number $p>0$. For an integer $n>0$, let $R_n:=\mathbb{Z}_p[[T]][p^{\frac{1}{p^n}}, T^\frac{1}{p^n}]$, denote by $R_n[(\frac{p}{T})^\frac{1}{p^n}]$ the $R_n$-subalgebra $R_n[(T^\frac{1}{p^n})^{-1}p^\frac{1}{p^n}]$ of $R_n[({T^\frac{1}{p^n}})^{-1}]=R_n[\frac{1}{T}]$ and let $R_\infty[(\frac{p}{T})^\frac{1}{p^\infty}]$ be the ring $\bigcup_{n>0}R_n[(\frac{p}{T})^\frac{1}{p^n}]$. Since $R_n$ is a regular local ring with a regular system of parameters $p^\frac{1}{p^n}, T^\frac{1}{p^n}$, one can easily show that $R_n[(\frac{p}{T})^\frac{1}{p^n}]$ is completely integrally closed in $R_n[\frac{1}{T}]$. Hence $R_\infty[(\frac{p}{T})^\frac{1}{p^\infty}]$ is completely integrally closed in $R_\infty[(\frac{p}{T})^\frac{1}{p^\infty}][\frac{1}{T}]$. Denote by $\widehat{R_\infty}\langle(\frac{p}{T})^\frac{1}{p^\infty}\rangle$ the $T$-adic completion of $R_\infty[(\frac{p}{T})^\frac{1}{p^\infty}]$, and let $A$ be a Banach ring associated to $\big(\widehat{R_\infty}\langle(\frac{p}{T})^\frac{1}{p^\infty}\rangle, (T)\big)$. Then $A$ is uniform and $A^\circ=\widehat{R_\infty}\langle(\frac{p}{T})^\frac{1}{p^\infty}\rangle$ by Lemma \ref{Bhlem}. 
Hence $A$ is a perfectoid ring in the sense of Fontaine by putting $\varpi=T^{\frac{1}{p}}$. Indeed, we have $p \in \varpi^p A^\circ$. In this example, let us point out that if $x \in A^\circ$ satisfies $\varpi^px=p$, then $x$ is not a unit.
\end{example}

Fix a prime number $p>0$ and let us recall the definition of \textit{Witt-perfect rings} due to Davis and Kedlaya as in \cite{DK14} and \cite{DK15}.

\begin{definition}[Witt-perfect ring]
Let $p>0$ be a prime number. Then we say that a ring $A_0$ is \textit{Witt-perfect}, if the Witt-Frobenius map $ \mathbf{F}:\mathbf{W}_{p^n}(A_0) \to \mathbf{W}_{p^{n-1}}(A_0)$ is surjective for all $n \ge 2$. 
\end{definition}

For the rest of the paper, we often assume that $A_0$ is $p$-torsion free. We describe the relationship between Witt-perfect and perfectoid algebras in Proposition \ref{decompperf} and Proposition \ref{perfalgdecomp}. These results say that Witt-perfect rings are viewed as ``decompletions'' of perfectoid algebras under a kind of preuniformity assumptions (cf.\ Definition \ref{defpair}). Our proof is based on the following criterion for being Witt-perfect.

\begin{lemma}[{\cite[Theorem 3.2]{DK14}}]
\label{PerfWitt}
For a prime number $p>0$, assume that $A_0$ is a $p$-torsion free ring. Then the following statements are equivalent.
\begin{enumerate}
\item
$A_0$ is a Witt-perfect ring.

\item
The Frobenius endomorphism on $A_0/(p)$ is surjective and for every $a \in A_0$, one can find $b \in A_0$ such that $b^p \equiv pa \pmod{p^2 A_0}$. 
\end{enumerate}
\end{lemma}

\begin{proposition}
\label{decompperf}
Let $A_0$ be a $p$-torsion free ring. Denote by $\widehat{A_0}$ the $p$-adic completion of $A_0$. 
\begin{enumerate}
\item
The following conditions are equivalent. 
\begin{itemize}
\item[$(a)$]{$A_0$ is Witt-perfect and integrally closed (resp.\ completely integrally closed) in $A_0[\frac{1}{p}]$. }
\item[$(b)$]{For any Banach ring $R$ associated to $(\widehat{A_0}, (p))$ 
(cf.\ Definition \ref{defbanach}(5)), $R$ is perfectoid in the sense of Fontaine and $\widehat{A_0}$ is open and integrally closed in $R$ (resp.\ $R^\circ=\widehat{A_0}$).}
\end{itemize}
\item
Assume further that $A_0$ is a $p$-adically separated valuation ring. Then the following conditions are equivalent. 
\begin{itemize}
\item[$(a)$]{$A_0$ is Witt-perfect and of rank $1$. }
\item[$(b)$]{For any Banach ring $K$ associated to $(\widehat{A_0}, (p))$, $K$ is a perfectoid field and $K^\circ=\widehat{A_0}$.}
\end{itemize}
\end{enumerate}
\end{proposition}

\begin{proof}
$(1)$: Let $R$ be a Banach ring associated to $(\widehat{A_0}, (p))$. First we assume that $A_0$ is Witt-perfect and integrally closed in $A_0[\frac{1}{p}]$. Then $\widehat{A_0}$ is integrally closed in $R$ by Corollary \ref{Bhlem}, $\widehat{A_0}/(p)$ is semiperfect, and there is some $\varpi\in\widehat{A_0}$ such that $\varpi^p\equiv p\mod p^2\widehat{A_0}$: In particular, $p=\varpi^pu$ holds for some unit $u \in \widehat{A_0}$, because $\widehat{A_0}$ is $p$-adically Zariskian. Thus we also have some $t\in \widehat{A_0}$ for which $t^p\equiv \varpi \pmod{p\widehat{A_0}}$ (and so $\varpi=t^pu'$ for some unit $u'\in \widehat{A_0}$). 
Hence by Lemma \ref{lem2231930}, $R^\circ/(p)$ is semiperfect. Therefore, $R$ is perfectoid. If further $A_0$ is completely integrally closed in $A_0[\frac{1}{p}]$, then $\widehat{A_0}$ is completely integrally closed in $R$ by Corollary \ref{Bhlem}, and so $R^\circ=\widehat{A_0}$. Consequently we obtain the implication $(a)\Rightarrow (b)$. 
Conversely, we then assume the condition $(b)$ (i.e.\ $R$ is perfectoid and $\widehat{A_0}$ is open and integrally closed in $R$). 
Then $A_0$ is integrally closed in $A$ by Corollary \ref{Bhlem}. Moreover, $\widehat{A_0}/(p) \cong A_0/(p)$ is semiperfect and there is some $t\in \widehat{A_0}$ for which $t^p \equiv p \pmod{p^2\widehat{A_0}}$ by Lemma \ref{1151321}(2). Hence $A_0$ is Witt-perfect. If further $R^\circ=\widehat{A_0}$, then $A_0$ is completely integrally closed in $A_0[\frac{1}{p}]$ by Corollary \ref{Bhlem}. 
Consequently we find that $(b)$ implies $(a)$, as required. 

$(2)$: First we assume $(a)$. Then $\widehat{A_0}$ is a valuation ring of rank $1$ with the fraction field $\widehat{A_0}[\frac{1}{p}]$. 
Thus, for a norm $||\cdot||$ on $\widehat{A_0}[\frac{1}{p}]$ associated with $(\widehat{A_0}, (p))$, the spectral norm $||\cdot||_\textnormal{sp}$ is multiplicative by Lemma \ref{valuation-norm}. We equip $\widehat{A_0}[\frac{1}{p}]$ with the norm $||\cdot||_\textnormal{sp}$. 
Then by the assertion (1), we find that $\widehat{A_0}[\frac{1}{p}]$ is perfectoid (and so it is a perfectoid field) and $(\widehat{A_0}[\frac{1}{p}])^\circ=\widehat{A_0}$. Hence $(b)$ follows. 
Next we assume $(b)$, conversely. Then in view of (1), $A_0$ is Witt-perfect and completely integrally closed in $A_0[p^{-1}]$. Therefore, the value group is of rank $1$. Hence $(a)$ follows.
\end{proof}

\begin{corollary}
\label{DecompTopNil}
Let $A_0$ be a $p$-torsion free Witt-perfect ring that is integrally closed in $A_0[\frac{1}{p}]$. 
Denote by $\widehat{A_0}$ the $p$-adic completion of $A_0$. Denote by $I$ and $I'$ the ideals $\sqrt{(p)}\subset A_0$ and $\sqrt{(p)}\subset \widehat{A_0}$, respectively. 
Then one has $I'=I\widehat{A_0}$ and $I=I^2$. 
\end{corollary}

\begin{proof}
By Lemma \ref{PerfWitt}, there is a sequence $\{\varpi_n\}_{n\geq 0}$ in $A_0$ such that $\varpi_0=p$, $\varpi_{1}^p\equiv \varpi_0\mod (p^2)$, and $\varpi^p_{n+1}\equiv \varpi_n\mod (p)$ for every $n\geq 0$. By induction on $n$, we have $\varpi^p_{n+1}=\varpi_nu_n$ for some unit $u_n\in \widehat{A_0}$, because $I'$ is contained in the Jacobson radical of $\widehat{A_0}$. Let $R$ be a Banach ring associated to $(\widehat{A_0}, (p))$. Then $R^{\circ\circ}=I'$. Moreover, by Lemma \ref{preunipm}(2) and Lemma \ref{decompperf}(1), we may assume that the norm on $R$ is powermultiplicative. Hence $I'$ is generated by $\{\varpi_n\}_{n\geq 0}$. Therefore, one has $I'=I\widehat{A_0}$ and $I'={I'}^2$. In particular, $I\widehat{A_0}=I^2\widehat{A_0}$. Next pick an element $x\in I$. Then by the equality stated just now, we have $x=\sum_{i=1}^r y_i\alpha_i$ for some $y_i\in I^2$ and $\alpha_i\in \widehat{A_0}$ $(i=1,\ldots, r)$. Take $a_i\in A_0$ for which $a_i\equiv \alpha_i \pmod{p^2 \widehat{A_0}}$ $(i=1,\ldots,r)$. Then we have $x-\sum^r_{i=1}y_ia_i\in p^2\widehat{A_0}\cap A_0=p^2A_0$. Therefore, $x\in I^2$, as wanted. 
\end{proof}

\subsection{Almost Witt-perfect and almost perfectoid algebras}

Let us recall Andr\'e's almost perfectoid algebras (cf.\ Definition 3.5.2 and Proposition 3.5.4 in \cite{An1}).

\begin{definition}[Almost perfectoid $K\langle T^\frac{1}{p^\infty}\rangle$-algebra]
\label{defalmostperf}
Let $K$ be a perfectoid field and let $R$ be a uniform Banach $K\langle T^\frac{1}{p^\infty}\rangle$-algebra. Let $\fm:=(T)^\frac{1}{p^\infty}K^{\circ\circ}K^\circ\langle T^\frac{1}{p^\infty}\rangle$ be an ideal of $K^\circ\langle T^\frac{1}{p^\infty}\rangle$. We say that $R$ is an \emph{almost perfectoid $K\langle T^\frac{1}{p^\infty}\rangle$-algebra}, if the Frobenius endomorphism on $R^\circ/(p)$ is $\fm$-almost surjective. 
\end{definition}

We then introduce the following class of rings to establish a variant of Lemma \ref{decompperf} fitting for almost mathematics.

\begin{definition}[Almost Witt-perfect ring]
\label{AlmostWittPerf}
Let $A_0$ be a $p$-torsion free ring with an element $g \in A_0$ admitting a compatible system of $p$-power roots $g^{\frac{1}{p^n}} \in A_0$. Then we say that $A_0$ is \textit{$(g)^{\frac{1}{p^\infty}}$-almost Witt-perfect}, if the following conditions are satisfied. 
\begin{enumerate}
\item
The Frobenius endomorphism on $A_0/(p)$ is $(g)^{\frac{1}{p^\infty}}$-almost surjective.

\item
For every $a\in A_0$ and every $n>0$, there is an element $b \in A_0$ such that $b^p \equiv pg^\frac{1}{p^n}a \pmod{p^2A_0}$. 
\end{enumerate}
\end{definition}

\begin{proposition}
\label{perfalgdecomp}
Let $V$ be a $p$-adically separated $p$-torsion free valuation domain and let $A_0$ be a $p$-torsion free $V[T^\frac{1}{p^\infty}]$-algebra. 
Put $g^\frac{1}{p^n}:=T^\frac{1}{p^n}\cdot 1\in A_0$ for every $n\geq 0$ and denote by $\widehat{V}$ and $\widehat{A_0}$ the $p$-adic completions of $V$ and $A_0$, respectively. Then the following conditions are equivalent. 

\begin{itemize}
\item[$(a)$]
$V$ is a Witt-perfect valuation domain of rank $1$ and $A_0$ is  $(g)^\frac{1}{p^\infty}$-almost Witt-perfect and integrally closed (resp.\ completely integrally closed) in $A_0[\frac{1}{p}]$. 
\item[$(b)$]
There exist a perfectoid field $K$ and an almost perfectoid $K\langle T^\frac{1}{p^\infty}\rangle$-algebra $R$ with the following properties: 
\begin{itemize}
\item[$\bullet$]
$K$ is a Banach ring associated to $(\widehat{V}, (p))$, the norm on $K$ is multiplicative and $K^\circ=\widehat{V}$;
\item[$\bullet$]
$R$ is a Banach ring associated to $(\widehat{A_0}, (p))$, and $\widehat{A_0}$ is open and integrally closed in $R$ (resp.\ $R^\circ=\widehat{A_0}$);
\item[$\bullet$]
the bounded homomorphism $K\langle T^\frac{1}{p^\infty}\rangle\to R$ is induced by the ring map $V[T^\frac{1}{p^\infty}]\to A_0$.  
\end{itemize}
\end{itemize}
\end{proposition}

\begin{proof}
We first prove $(a) \Rightarrow (b)$. We denote by $||\cdot||_1$ (resp.\ $||\cdot||_2$) the norm $||\cdot||_{\widehat{V}, (p), p}$ on $\widehat{V}[\frac{1}{p}]$ (resp.\ the norm $||\cdot||_{\widehat{A_0}, (p), p}$ on $\widehat{A_0}[\frac{1}{p}]$), and let $||\cdot||_{1, \textnormal{sp}}$ (resp.\ $||\cdot||_{2, \textnormal{sp}}$) be the associated spectral seminorm. Notice that $||\cdot||_{1, \textnormal{sp}}$ (resp.\ $||\cdot||_{2, \textnormal{sp}}$) is equivalent to $||\cdot||_{\widehat{V}, (p), p}$ (resp.\ $||\cdot||_{\widehat{A_0}, (p), p}$) by Corollary \ref{Bhlem} and Lemma \ref{preunipm}(2). We then equip $\widehat{V}[\frac{1}{p}]$ with $||\cdot||_{1,\textnormal{sp}}$ 
(resp.\ $\widehat{A_0}[\frac{1}{p}]$ with $||\cdot||_{2, \textnormal{sp}}$), and denote by $K$ (resp.\ $R$) the resulting Banach ring. 
Then the ring map $V[T^\frac{1}{p^\infty}]\to {A_0}$ induces a bounded map $K\langle T^\frac{1}{p^\infty}\rangle\to R$. 
Moreover, $K$ is a perfectoid field with $K^\circ=\widehat{V}$ by Proposition \ref{decompperf}(2), and $\widehat{A_0}$ forms an open and integrally closed subring of $R$ by Corollary \ref{Bhlem}. Let $\varpi\in K^\circ$ be a perfectoid pseudouniformizer that satisfies $p=\varpi^pu$ for some unit $u\in K^\circ$ and admits a compatible system of $p$-power roots (such an element $\varpi$ exists by Lemma \ref{1151321} and \cite[Lemma 3.9]{BMS17}). Then, since the Frobenius endomorphism on $\widehat{A_0}/(p)$ is $(g)^\frac{1}{p^\infty}$-almost surjective, the Frobenius endomorphism on $R^\circ/(p)$ is $(\varpi g)^\frac{1}{p^\infty}$-almost surjective by Corollary \ref{Bhlem}, Corollary \ref{gAl-pigAl} and Lemma \ref{lem2231930}. So we conclude that $R$ is an almost perfectoid $K\langle T^\frac{1}{p^\infty}\rangle$-algebra. The remaining part follows from Corollary \ref{Bhlem} and Lemma \ref{lemma10092}.

Next we prove $(b)\Rightarrow (a)$. Assume $(b)$. In view of Corollary \ref{Bhlem} and Proposition \ref{decompperf}(2), it suffices to show that $A_0$ is $(g)^\frac{1}{p^\infty}$-almost Witt-perfect. Let $\varpi\in K^\circ$ be a perfectoid pseudouniformizer with the property mentioned above. Since $K^\circ=\widehat{V}$ and the map $K\to R$ carries $\widehat{V}$ into $\widehat{A_0}$, there is some unit $u\in \widehat{A_0}$ such that $p=\varpi^p u$. Moreover, the Frobenius endomorphism on $R^\circ/(p)$ is $(\varpi g)^\frac{1}{p^\infty}$-almost surjective by assumption. 
Hence by Corollary \ref{gAl-pigAl} and Lemma \ref{lem2231930}, the Frobenius endomorphism on $\widehat{A_0}/(p) \cong A_0/(p)$ is $(g)^\frac{1}{p^\infty}$-almost surjective. Thus it is enough to check the condition (2) in Definition \ref{AlmostWittPerf} for $a=1$. Fix an integer $n>0$. Then we have $g^\frac{1}{p^n}p=\varpi^p(g^\frac{1}{p^n}u)$, and there exist some $b, c\in \widehat{A_0}$ such that $g^\frac{1}{p^n}u=b^p+pc$. Thus we have $g^\frac{1}{p^n}p=(\varpi b)^p+\varpi^ppc=(\varpi b)^p+p^2u^{-1}c$, which yields $(\varpi b)^p \equiv pg^\frac{1}{p^n} \pmod{p^2\widehat{A_0}}$. Since $\widehat{A_0}/(p^2) \cong A_0/(p^2)$, the assertion follows.
\end{proof}

\section{Finite \'etale extensions}
\subsection{Finite \'etale extensions and completeness}

Let $A_{0}$ be a ring, let $I_{0}\subset A_{0}$ be an ideal, and let $M_{0}$ be an $A_{0}$-module. Assume that $A_{0}$ is $I_{0}$-adically complete and separated. Then one may ask the question: Is $M_{0}$ also $I_{0}$-adically complete and separated? As is well known, if $A_{0}$ is Noetherian and $M_{0}$ is a finitely generated $A_0$-module, then the question is affirmative; see \cite[Theorem 8.7]{Ma86}. However, in the absence of Noetherian property, this question is subtle and often require a careful argument (or even counterexamples exist). Now we consider the following case: there exists a ring extension $A_{0}\subset A$ such that $M_{0}$ is an $A_{0}$-submodule of some finite projective $A$-algebra $B$. Then in some situations, the completeness is ensured by a condition on the trace map (even if $M_{0}$ or $I_0$ is not finitely generated). A detailed account of the trace map for finite projective ring extensions is found in \cite{Fo17}.

\begin{proposition}
\label{prop513}
Let $A$ be a ring and let $B$ be a finite \'etale $A$-algebra. Let $A_{0}\subset A$ be a subring with an ideal $I_{0}\subset A_{0}$. Let $B_{0}$ be an $A_{0}$-subalgebra of $B$ such that $B=\bigcup_{n\geq 1}(B_{0}:_{B}I_{0}^{n})$. Assume that there exists an integer $c>0$ such that 
$\textnormal{Tr}_{B/A}(tm)\in A_{0}$ for every $t\in I_0^c$ and every $m\in B_{0}$. 

\begin{enumerate}
\item
If $A_{0}$ is $I_{0}$-adically separated, then so is $B_{0}$. 

\item
If $A_{0}$ is $I_{0}$-adically complete, then so is $B_{0}$. 
\end{enumerate}
\end{proposition}

\begin{proof}
First note that since $B$ is finite \'etale over $A$, the $A$-homomorphism
\begin{eqnarray}\label{lem1223}
B\to \textnormal{Hom}_{A}(B, A),\ b\mapsto\textnormal{Tr}_{B/A}(b \cdot\ )
\end{eqnarray}
is an isomorphism by \cite[Corollary 4.6.8]{Fo17}. 

Let us prove $(1)$. Pick $m\in \bigcap_{n=0}^{\infty}I_{0}^{n}B_0$. Since (\ref{lem1223}) is injective, it suffices to show that $\textnormal{Tr}_{B/A}(mx)=0$ for an arbitrary element $x\in B$. Take $l>0$ for which $I^{l}_{0}x\subset B_{0}$. Then for every $n>0$, we have $m \in I_0^{n+c+l}B_0$ and thus, there exist $t_{n, i}\in I^{n}_{0}$, $u_{n,i}\in I^{c+l}_{0}$, and $m_{n,i}\in B_{0}$ ($i=1,\ldots,r$) such that $m=\sum^{r}_{i=1}t_{n, i}u_{n,i}m_{n, i}$. Hence
\begin{eqnarray}
\textnormal{Tr}_{B/A}(mx)=\sum^{r}_{i=1}t_{n,i}\textnormal{Tr}_{B/A}(u_{n,i}xm_{n,i})
\in I^{n}_{0}A_0\ \nonumber
\end{eqnarray}
for every $n>0$. Thus, since $A_{0}$ is $I_{0}$-adically separated, we have $\textnormal{Tr}_{B/A}(mx)=0$, as desired. 

Next we prove $(2)$. Since $B$ is a finite projective $A$-module, there exist $A$-homomorphisms $s: B\to A^{\oplus d}$ and $\pi: A^{\oplus d}\to B$ such that $\pi\circ s$ is the identity.  Let us equip $A$ (resp.\  $A^{\oplus d}$, resp.\ $B$) with the topology such that $\{I_0^nA_0\}_{n\geq 1}$ (resp.\ $\{I_{0}^{n}A_{0}^{\oplus d}\}_{n\geq 1}$, resp.\ $\{I_0^nB_0\}_{n\geq 1}$) forms a system of fundamental open neighborhoods of $0$. We consider these topologies in what follows. Since (\ref{lem1223}) is surjective, there exist $b_{1},\ldots, b_{d}\in B$ such that $s(x)=(\textnormal{Tr}_{B/A}(b_{1}x),\ldots,\textnormal{Tr}_{B/A}(b_{d}x))$ for every $x\in B$. Let $\{m_{n}\}_{n\geq 1}$ be a Cauchy sequence in $B_{0}$ with respect to the $I_0$-adic topology. Then for each $i=1,\ldots,d$, $\{\textnormal{Tr}_{B/A}(b_{i}m_{n})\}_{n\geq 1}$ forms a Cauchy sequence in $A$. Hence by assumption, each $\{\textnormal{Tr}_{B/A}(b_{i}m_{n})\}_{n\geq 1}$ converges to some $a_{i}\in A$. Then $\{s(m_{n})\}_{n\geq 1}$ converges to $(a_{1},\ldots, a_{d})\in A^{\oplus d}$. Thus, $\{m_{n}\}_{n\geq 1}=\{\pi(s(m_{n}))\}_{n\geq 1}$ converges to $\pi((a_{1},\ldots, a_{d})) \in B$. Since $\pi((a_{1},\ldots, a_{d}))-m_n\in B_0$ for $n \gg 0$, we have $\pi((a_{1},\ldots, a_{d})) \in B_0$. Hence the assertion follows. 
\end{proof}

\begin{proposition}
\label{prop1224}
Let $A_0$ be a ring with a nonzero divisor $t$ and put $A=A_0[\frac{1}{t}]$. Let $B$ be a finite \'etale $A$-algebra. Let $B_{0} \subset B$ be an $A_{0}$-subalgebra for which $B=B_{0}[\frac{1}{t}]$. Assume that there exists some $l>0$ such that $\textnormal{Tr}_{B/A}(t^lb)\in A_{0}$ for every $b\in B_0$. Denote by $\widehat{A_0}$  and $\widehat{B_0}$  the $t$-adic completions of $A_0$ and $B_0$, respectively. Then the natural $A_0$-algebra homomorphism $B_{0}\otimes_{A_{0}}\widehat{A_{0}} \to \widehat{B_0}$ induces an isomorphism:
\begin{eqnarray}
\label{eq1226}
(B_{0}\otimes_{A_{0}}\widehat{A_{0}})/(0)^{t-\rm{sat}} \xrightarrow{\cong} \widehat{B_0} 
\end{eqnarray}
(where $(0)^{t\textnormal{-sat}}$ denotes the $(t)$-saturation of the ideal $(0)\subset B_{0}\otimes_{A_{0}}\widehat{A_{0}}$). 
In particular, the natural $A$-algebra homomorphism 
$$
(B_{0}\otimes_{A_{0}}\widehat{A_{0}})[\frac{1}{t}] \to(\widehat{B_{0}})[\frac{1}{t}]
$$
is an isomorphism. 
\end{proposition}

\begin{proof}
Since $\widehat{B_0}$ is $t$-torsion free, the map $\varphi: B_{0}\otimes_{A_{0}}\widehat{A_{0}} \to \widehat{B_0}$ induces a commutative diagram: 
\[\xymatrix{
B_0\otimes_{A_0}\widehat{A_0}\ar[rr]^\pi\ar[rrd]_{\varphi}&&(B_{0}\otimes_{A_{0}}\widehat{A_{0}})/(0)^{t-\rm{sat}}\ar[d]^{\widetilde{\varphi}}\\
&&\widehat{B_0}
}\]
where $\pi$ is the canonical projection map. We prove that $\widetilde{\varphi}$ is an isomorphism. 
First we show that $(B_{0}\otimes_{A_{0}}\widehat{A_{0}})/(0)^{t-\rm{sat}}$ is $t$-adically complete and separated. 
Let us apply Proposition \ref{prop513} by setting $I_0=(t)$. Put $A':=(\widehat{A_{0}})[\frac{1}{t}]$. 
Notice that $(B_{0}\otimes_{A_{0}}\widehat{A_{0}})/(0)^{t-\rm{sat}}$ is isomorphic to the $\widehat{A_{0}}$-subalgebra $C_0\subset (B_{0}\otimes_{A_{0}}\widehat{A_{0}})[\frac{1}{t}]$ that is the image of $B_{0}\otimes_{A_{0}}\widehat{A_{0}}\to(B_{0}\otimes_{A_{0}}\widehat{A_{0}})[\frac{1}{t}]$, and $(B_{0}\otimes_{A_{0}}\widehat{A_{0}})[\frac{1}{t}]$ is identified with the finite \'etale $A'$-algebra $B\otimes_{A}A'$. Since we have
$$
\textnormal{Tr}_{B\otimes_{A}A'/A'}(b\otimes_{A} 1_{A'})=\textnormal{Tr}_{B/A}(b)\otimes_{A} 1_{A'}~(\forall b\in B),
$$
it follows that every element $c\in C_{0}$ satisfies $\textnormal{Tr}_{B\otimes_{A}A'/A'}(t^lc)\in \widehat{A_{0}}$. Hence $C_{0}$ is $t$-adically complete and separated by Proposition \ref{prop513}, and therefore so is $(B_{0}\otimes_{A_{0}}\widehat{A_{0}})/(0)^{t-\rm{sat}}$, as wanted. 
Now note that the $t$-adic completion 
$\widehat{B_{0}\otimes_{A_{0}}\widehat{A_{0}}}$ of $B_{0}\otimes_{A_{0}}\widehat{A_{0}}$ is naturally isomorphic to $\widehat{B_0}$. 
Then by the universality of completion \cite[Proposition 7.1.9 in Chapter 0]{FK18}, the isomorphism $\widehat{B_{0}\otimes_{A_{0}}\widehat{A_{0}}} \xrightarrow{\cong} \widehat{B_0}$ factors as
$$
\widehat{B_{0}\otimes_{A_{0}}\widehat{A_{0}}} \xrightarrow{\widehat{\pi}} (B_{0}\otimes_{A_{0}}\widehat{A_{0}})/(0)^{t-\rm{sat}} \xrightarrow{\widetilde{\varphi}}  \widehat{B_0},
$$
where the map $\widehat{\pi}$ is surjective, because $\pi$ is so. Since $\widetilde{\varphi}\circ \widehat{\pi}$ is an isomorphism, it follows that $\widehat{\pi}$ is an isomorphism. 
Thus, $\widetilde{\varphi}$ is also an isomorphism. Finally, It readily follows that $(B_{0}\otimes_{A_{0}}\widehat{A_{0}})[\frac{1}{t}] \to(\widehat{B_{0}})[\frac{1}{t}]$ is an isomorphism.
\end{proof}

As the condition on the trace map in Proposition \ref{prop1224} is subtle, we will unravel some verifiable hypotheses on ring maps with a desired trace map in the next subsection.

\subsection{Studies on preuniform pairs and the condition $(*)$}
\label{StudPreUniPair}
Here we establish several basic properties of preuniform pairs (cf.\ Definition \ref{defpair}). We especially investigate the following condition.

\begin{definition}\normalfont\label{CondQuad*}
Let $f_0: (A_0, I_0)\to (B_0, J_0)$ be a morphism of pairs. Then we say that $f_0$ satisfies ``the condition $(*)$", if $I_0=tA_0$ and $J_0=tB_0$ for some $t\in A_0$ and $f_0$ satisfies the following axioms: 
\begin{itemize}
\item[(a)]
$(A_0, I_0)$ is preuniform.

\item[(b)]
The ring map $A_0[\frac{1}{t}]\to B_0[\frac{1}{t}]$ induced by $f_0$ is finite \'etale. 

\item[(c)]
$B_0$ is $t$-torsion free, and $B_0\subset (A_0)^*_{B_0[\frac{1}{t}]}$. 
\end{itemize}
\end{definition}

\begin{example}
Let $V$ be a valuation ring with a non-zero element $t\in V$ for which $V$ is $t$-adically separated. 
Set $K:=\textnormal{Frac}(V)$. 
Then, as we observed in Example \ref{EgPreUni}(1), the pair $(V, (t))$ is preuniform and $K=V[\frac{1}{t}]$. 
Let $L$ be a finite separable extension of $K$, and  $W$ the integral closure of $V$ in $L$. 
Then any $x\in L$ is integral over $K$, and hence satisfies $t^{l}x\in W$ for some $l>0$ because $K=V[\frac{1}{t}]$. Therefore, $L=W[\frac{1}{t}]$. Hence the inclusion map $V\hookrightarrow W$ becomes finite \'etale after inverting $t$, and we have $W=V^{+}_{W[\frac{1}{t}]}\subset V^{*}_{W[\frac{1}{t}]}$. 
Thus the morphism $(V, (t))\to (W, (t))$ satisfies $(*)$.  
\end{example}

A morphism $(A_0, (t))\to (B_0, (t))$ satisfying $(*)$ has the following good properties. 
In particular, the condition imposed on the trace map of Proposition \ref{prop1224} may be realized by it.

\begin{proposition}
\label{cor1416}
Let $f_0: (A_0, (t))\to (B_0, (t))$ be a morphism of pairs that satisfies the condition $(*)$. 
Put $A:=A_0[\frac{1}{t}]$ and $B:=B_0[\frac{1}{t}]$. 
Then the following assertions hold. 

\begin{enumerate}
\item
There exists an integer $c>0$ such that $\textnormal{Tr}_{B/A}(t^cB_0) \subset A_{0}$.

\item
There exist an integer $l>0$, a finite free $A_0$-module $F$, and $A_0$-homomorphisms $B_0\to F\to B_0$ whose composition is  multiplication by $t^l$. In particular, $t^lB_0$ is contained in a finitely generated $A_0$-submodule of $B_0$. 

\item
One has $(A_0)^*_B=(B_0)^*_B$. 

\item
The pair $(B_0, (t))$ is preuniform. 
\end{enumerate}
\end{proposition}

To prove this, we need the following lemma.

\begin{lemma}
\label{DK26}
Let $A_0$ be a ring with a nonzero divisor $t$, and put $A=A_0[\frac{1}{t}]$. Let $B$ be a finite \'etale $A$-algebra. 
Assume that $A_0$ is integrally closed in $A$. Then for every integral element $b\in B$ over $A$, $\textnormal{Tr}_{B/A}(b)$ belongs to $A_0$. 
\end{lemma}

\begin{proof}[Proof of Lemma \ref{DK26}]
A proof of the lemma in the special case that $t$ is a prime number is given in \cite[Lemma 2.6]{DK15}, but the same proof is valid for the general case. 
\end{proof}

Here notice the following fact.

\begin{lemma}
\label{IntClos*Stable}
Keep the notation as in Proposition \ref{cor1416}. 
Put $A^+:=(A_0)^+_A$ and $B^+:=(A_0)^+_B$. 
Let $f: A\to B$ be the ring map induced by $f_0$, and let $f^+: A^+\to B^+$ be the ring map such that $f|_{A^+}$ factors through $f^+$. 
Then $B=B^+[\frac{1}{t}]$, and the morphism $f^+: (A^+, (t))\to (B^+, (t))$ satisfies the condition $(*)$. 
\end{lemma}

\begin{proof}[Proof of Lemma \ref{IntClos*Stable}]
It is clear because $B$ is integral over $A$ and $A=A_0[\frac{1}{t}]$. 
\end{proof}

Now let us start to prove Proposition \ref{cor1416}.

\begin{proof}[Proof of Proposition \ref{cor1416}]
We define the morphism $f^+: (A^+, (t))\to (B^+, (t))$ as in Lemma \ref{IntClos*Stable}.   
Then by Lemma \ref{DK26}, we have $\textnormal{Tr}_{B/A}(B^+)\subset A^+$. 
Thus, since $(A_0, (t))$ is preuniform, there exists some $c_1>0$ such that $\textnormal{Tr}_{B/A}(t^{c_1}B^+)\subset A_0$. Meanwhile, we have $tB_0\subset t(A_0)^*_B\subset B^+$ by Lemma \ref{AlmIntMod-Fin}. 
Thus it holds that $\textnormal{Tr}_{B/A}(t^{c_1+1}B_0)\subset A_0$, which yields the assertion $(1)$.

Let us prove $(2)$. Since $B$ is finite projective over $A$, we have $A$-homomorphisms $\pi: A^{\oplus d}\to B$ and $s: B\to A^{\oplus d}$ such that $\pi\circ s$ is the identity. Now since (\ref{lem1223}) is surjective, there exist $b_{1},\ldots, b_{d}\in B$ such that $s(x)=(\textnormal{Tr}_{B/A}(b_{1}x),\ldots,\textnormal{Tr}_{B/A}(b_{d}x))$ for every $x\in B$. Hence by (1), there exists $l_1>0$ such that $(t^{l_1}s)|_{B_0}$ factors through an $A_0$-homomorphism $s_{l_1}: B_0\to A_0^{\oplus d}$. Now there also exists an integer $l_2>0$ such that $(t^{l_2}\pi)|_{A_0^{\oplus d}}$ factors through an $A_0$-homomorphism $\pi_{l_2}: A_0^{\oplus d}\to B_0$. Then $\pi_{l_2}\circ s_{l_1}$ is multiplication by $t^{l_1+l_2}$, which yields the claim. 

Next we prove $(3)$. The containment $(A_0)^*_B\subset (B_0)^*_B$ is easy to see. Let us show the reverse inclusion. Pick an element $b\in B$ and assume that $b$ is almost integral over $B_0$. Then, since $B=B_0[\frac{1}{t}]$, there exists some $m>0$ such that $t^m(\sum^\infty_{n=0}B_0\cdot b^n)\subset B_0$. 
Hence by the assertion $(2)$, $t^{l+m}(\sum^\infty_{n=0}B_0\cdot b^n)$ is contained in a finitely generated $A_0$-submodule of $B_0$. Therefore, $\sum^\infty_{n=0}A_0\cdot b^n$ is contained in a finitely generated $A_0$-submodule of $B$. Hence $b\in B$ is almost integral over $A_0$, as desired.

Finally we prove $(4)$. By the assertion $(3)$, we have $(B_0)^+_B\subset (B_0)^*_B\subset (A_0)^*_B$. Hence in view Lemma \ref{AlmIntMod-Fin}, we have $t(B_0)^+_B\subset t(A_0)^*_B\subset (A_0)^+_B$. On the other hand, by Lemma \ref{IntClos*Stable} and the assertion $(2)$, there exists some $l'>0$ such that $t^{l'}(A_0)^+_B$ is contained in a finitely generated $(A_0)^+_A$-submodule $N_0$ of $(A_0)^+_B$. Moreover, since $(A_0, (t))$ is preuniform and $B=B_0[\frac{1}{t}]$, 
we have $t^{c'}N_0\subset B_0$ for some $c'>0$. Thus, $t^{c'+l'+1}(B_0)^+_B$ is contained in $B_0$. This yields the assertion. 
\end{proof}

\begin{corollary}
\label{cor1547205}
Let $A$ be a preuniform Tate ring. Let $f: A\to B$ be a finite \'etale ring map. Equip $B$ with the canonical structure as a Tate ring (cf.\ Lemma \ref{tatefinex}). Then the following assertions hold. 

\begin{enumerate}
\item
$B$ is also preuniform. In particular, for any ring of definition $A_0$ of $A$, $(A_0)^+_B$ and $(A_0)^*_B$ are rings of definition of $B$.  
\item
For any topologically nilpotent unit $t\in A$, the morphism $(A^\circ, (t))\to (B^\circ, (t))$ satisfies the condition $(*)$. 
\item
$B^\circ=(A^\circ)^*_B$. 
\item
If $f$ is injective, then $f^{-1}(B^\circ)=A^\circ$. 
\item
If $A$ is uniform, then so is $B$. 
\end{enumerate}
\end{corollary}

\begin{proof}
The assertions $(1)$, $(2)$ and $(3)$ are immediate consequences of Lemma \ref{AlmIntMod-Fin}. Lemma \ref{tatefinex} and Proposition \ref{cor1416}. The assertion (5) follows from (1) and Proposition \ref{prop513}. Let us prove (4). Assume that $f$ is injective. Since the containment $A^\circ\subset f^{-1}(B^\circ)$ is clear, it suffices to show the reverse inclusion. Pick $a\in A$ for which $f(a)\in B^\circ$. 
Then, since $B^\circ=(A^\circ)^*_B$ by the assertion (3), there is some $l>0$ such that $t^la^n\in A^\circ$ for every $n>0$. 
Hence we have $a\in (A^\circ)^*_A=A^\circ$, as wanted. 
\end{proof}

In the next theorem, we show that the condition $(*)$ is stable under completion. This property is quite important for our study (cf.\ Corollary \ref{IsomCompPowerBd} and Theorem \ref{lemalprjct}). The theorem itself can be interpreted as a practical form of Proposition \ref{prop1224}.

\begin{theorem}
\label{maincor1}
Keep the notation as in Proposition \ref{cor1416}. Denote by $\widehat{A_0}$ and $\widehat{B_0}$ the $t$-adic completions of $A_0$ and $B_0$, respectively. Let $\widehat{f_0}: (\widehat{A_0}, (t))\to (\widehat{B_0}, (t))$ be the morphism of pairs induced by $f_0$. Put $A':=\widehat{A_0}[\frac{1}{t}]$ and $B':=\widehat{B_0}[\frac{1}{t}]$. Then the following assertions hold. 

\begin{enumerate}
\item
The natural $A'$-algebra homomorphism $B\otimes_AA' \to B'$ is an isomorphism. 

\item
$\widehat{f_0}: (\widehat{A_0}, (t))\to (\widehat{B_0}, (t))$ also satisfies the condition $(*)$. 
\end{enumerate}
\end{theorem}

\begin{proof}
$(1)$ is a consequence of Proposition \ref{prop1224} and Proposition \ref{cor1416}(1). 
Let us prove $(2)$. By Proposition \ref{CICprop}, the morphism $\widehat{f_0}: (\widehat{A_0}, (t))\to (\widehat{B_0}, (t))$ satisfies (a) in Definition \ref{CondQuad*}. 
Moreover, as a corollary of the assertion $(1)$, one finds that $\widehat{f_0}$ also satisfies (b). Hence the remaining part is to show that $\widehat{B_0}$ is contained in $(\widehat{A_0})^*_{B'}$. 
To carry out this, we take a finitely generated $A_0$-submodule $N_0\subset B_0$ such that $t^lB_0\subset N_0$ for some $l>0$ by applying Proposition \ref{cor1416}(2). Denote by $\widehat{N_0}$ the $t$-adic completion of $N_0$. 
Then by Lemma \ref{CompCokerKilled}, $\widehat{N_0}$ is viewed as an $\widehat{A_0}$-submodule of $\widehat{B_0}$ such that 
\begin{equation}
\label{inclusion1}
t^l\widehat{B_0}\subset \widehat{N_0}.
\end{equation} 
By applying the topological Nakayama's lemma \cite[Theorem 8.4]{Ma86}, one finds that 
\begin{equation}
\label{inclusion2}
\widehat{N_0}~\mbox{is a finitely generated}~\widehat{A_0}\mbox{-module},
\end{equation}
because $\widehat{N_0}/t\widehat{N_0} \cong {N_0}/t{N_0}$ is finitely generated over $\widehat{A_0}/t\widehat{A_0} \cong A_0/tA_0$. Combining $(\ref{inclusion1})$ and $(\ref{inclusion2})$ together, we conclude that $\widehat{B_0}$ is contained in a finitely generated $\widehat{A_0}$-submodule of $B'$. In particular it holds that $\widehat{B_0}\subset (\widehat{A_0})^*_{B'}$, as wanted. 
\end{proof}

\begin{corollary}
\label{IsomCompPowerBd}
Let $(A_0, (t))$ be a preuniform pair with the associated Tate ring $A$. Let $\widehat{A_0}$ be the $t$-adic completion of $A_0$ and let $\mathcal{A}$ be the Tate ring associated to $(\widehat{A_0}, (t))$. Suppose that $B$ is a finite \'etale $A$-algebra and denote by $\mathcal{B}$ the finite \'etale $\mathcal{A}$-algebra $B\otimes_A\mathcal{A}$. Equip $B$ and $\mathcal{B}$ with the canonical structure as a Tate ring (cf.\ Lemma \ref{tatefinex}) respectively. Then the following assertions hold. 
\begin{enumerate}
\item
Let $\widehat{B^\circ}$ be the $t$-adic completion of $B^\circ$. 
Then the natural ring map $\varphi: \mathcal{B}\rightarrow \widehat{B^\circ}[\frac{1}{t}]$ is an isomorphism which induces an isomorphism 
$\mathcal{B}^\circ\xrightarrow{\cong}\widehat{B^\circ}$. 
\item
For the natural map $\psi: B\to \mathcal{B}$, it holds that $\psi^{-1}(\mathcal{B}^\circ)=B^\circ$. 
\end{enumerate}
\end{corollary}

\begin{proof}
We first prove $(1)$. Let $\widehat{A^\circ}$ be the $t$-adic completion of $A^\circ$ and put $A':=\widehat{A^\circ}[\frac{1}{t}]$ and $B':=\widehat{B^\circ}[\frac{1}{t}]$. Then the natural ring map $\mathcal{A}\to A'$ is an isomorphism by Lemma \ref{CompCokerKilled}, and it induces $\widehat{A^\circ}\xrightarrow{\cong} \mathcal{A}^\circ$ by Corollary \ref{Bhlem}. 
Since the morphism $(A^\circ, (t))\to (B^\circ, (t))$ satisfies the condition $(*)$ by Corollary \ref{cor1547205}(2), $\varphi$ is an isomorphism 
by Theorem \ref{maincor1}(1). 
Thus, $\varphi$ induces an isomorphism of $\widehat{A^\circ}$-algebras $(\mathcal{A}^\circ)^*_\mathcal{B}\xrightarrow{\cong}(\widehat{A^\circ})^*_{B'}$. Therefore it suffices to show that 
\begin{equation}\label{equality36}
\mathcal{B}^\circ=(\mathcal{A}^\circ)^*_\mathcal{B}~\mbox{and}~ 
\widehat{B^\circ}=(\widehat{B^\circ})^{*}_{B'}=(\widehat{A^\circ})^*_{B'}\ . 
\end{equation}
By Theorem \ref{maincor1}(2), the morphism $(\widehat{A^\circ}, (t))\to (\widehat{B^\circ}, (t))$ also satisfies the condition $(*)$. 
Hence (\ref{equality36}) follows from Proposition \ref{cor1416}(3) and Corollary \ref{cor1547205}(3), as wanted. 

Let us prove the assertion $(2)$. Let $B'_0$ denote the subring $\psi^{-1}(\mathcal{B^\circ})$ of $B$. By Corollary \ref{cor1547205}(3), we have $B^\circ=(A^\circ)^*_B$ and $\mathcal{B^\circ}=(\mathcal{A}^\circ)^*_\mathcal{B}=(\widehat{A^\circ})^*_\mathcal{B}$. On the other hand, we have 
$$
\psi((A^\circ)^*_B)\subset (A^\circ)^*_\mathcal{B}\subset (\widehat{A^\circ})^*_\mathcal{B}.
$$ 
Thus it follows that $B^\circ\subset B'_0$. We then prove the reverse inclusion. By the assertion $(1)$, we have the commutative diagram: 
\begin{equation}\nonumber
\begin{CD}\label
{BLdiagram}
\mathcal{B}^\circ@>\cong>> \widehat{B^\circ} \\
@VVV @VVV \\
\mathcal{B} @>\cong>\varphi>\ \widehat{B^\circ}[\frac{1}{t}]
\end{CD}
\end{equation}
where the vertical arrows denote the inclusion maps. Hence the map $(\varphi\circ\psi)|_{B'_0}$ factors through $\widehat{B^\circ}$. 
Thus by Lemma \ref{Beauville-Laszlo}, we have $B'_0\subset B^\circ$, as wanted. 
\end{proof}

Finally, we remark the following result. Compare it with Lemma \ref{CompCokerKilled}.

\begin{lemma}
Let $(A_0, (t))$ be a preuniform pair and let $A_0 \hookrightarrow B_0$ be an integral ring extension such that $B_0$ is $t$-torsion free. Denote by $\widehat{A_0}$ and $\widehat{B_0}$ the $t$-adic completions of $A_0$ and $B_0$, respectively. Then $\widehat{A_0} \to \widehat{B_0}$ is an injective ring map between $t$-torsion free rings. 
\end{lemma}

\begin{proof}
The $t$-torsion freeness of $\widehat{A_0}$ and $\widehat{B_0}$ follows from Lemma \ref{Beauville-Laszlo}. Let us prove the injectivity of the homomorphism $\widehat{A_0} \to \widehat{B_0}$. Put $A=A_0[\frac{1}{t}]$ and choose an integer $c>0$ for which $t^c(A_0)^+_A\subset A_0$. Then it suffices to prove that $A_0 \cap t^{n+c}B_0\subset t^nA_0$ for every $n>0$. Let $x \in A_0 \cap t^{n+c}B_0$. Then one can write $x=t^{n+c} b$ for $b \in B_0$ and hence
$$
b=\frac{x}{t^{n+c}} \in B_0 \cap A.
$$
As $B_0$ is integral over $A_0$, we have $t^cb\in A_0$. Hence we have $x=t^n(t^cb) \in t^nA_0$, as required. 
\end{proof}

\section{The almost purity theorem for Witt-perfect rings}

\subsection{Almost \'etale ring maps}
First we review the definition of almost \'etale ring maps and the statement of the almost purity theorem, due to Davis and Kedlaya. For the results concerning classical \'etale ring maps, the recent book \cite{Fo17} is a good reference. The large part of almost ring theory is concerned around almost counterparts of certain finiteness conditions, so let us recall some definitions. The original definition of almost finitely generated (presented) modules is quite intricate; see \cite[Definition 2.3.8]{GR03}. However, the following characterization is more straightforward and found in \cite[Proposition 2.3.10]{GR03}.

\begin{definition}
Let $(R,I)$ be a basic setup, let $A$ be an $R$-algebra and let $M$ be an $A$-module. 
\begin{enumerate}
\item
$M$ is said to be \textit{$I$-almost finitely generated}, if for every finitely generated subideal $I_0 \subset I$, there exists a finitely generated $A$-submodule $N \subset M$ such that $I_0 M \subset N$.

\item
$M$ is said to be \textit{$I$-almost finitely presented}, if for every finitely generated subideal $I_0 \subset I$, there is a complex of $A$-modules: $A^{\oplus m} \xrightarrow{\psi} A^{\oplus n} \xrightarrow{\phi} M$ such that $I_0 \coker(\phi)=0$ and $I_0 \ker(\phi) \subset \im(\psi)$.

\end{enumerate}
\end{definition}

As our main concern is around \'etale ring maps and their almost variants, we refer the reader to \cite[Definition 2.4.4]{GR03} for the definitions of $I$-almost flatness and $I$-almost projectivity. We refer the reader to \cite[Definition 3.1.1]{GR03} for the following definitions.

\begin{definition}
Fix a basic setup $(R,I)$ and let $A \to B$ be an $R$-algebra homomorphism.
\begin{enumerate}
\item
$A \to B$ is called \textit{$I$-almost weakly unramified}, if the diagonal map $B\otimes_{A} B \to B$ is $I$-almost flat.

\item
$A\to B$ is called \textit{$I$-almost unramified}, if the diagonal map $B \otimes_{A} B \to B$ is $I$-almost projective.

\item
$A \to B$ is called \textit{$I$-almost weakly \'etale}, if $A \to B$ is $I$-almost flat and $I$-almost weakly unramified.

\item
$A \to B$ is called \textit{$I$-almost \'etale}, if $A\to B$ is $I$-almost flat and $I$-almost unramified.

\item
$A\to B$ is called \textit{$I$-almost finite \'etale}, if it is $I$-almost \'etale and $B$ is an $I$-almost finitely presented $A$-module.
\end{enumerate}
\end{definition}

So far, we have been using the symbol $(A_0,I_0)$ to emphasize that $A_0$ is a ring of definition of some Tate ring and $I_0$ is its ideal of definition.\footnote{The exceptional places are Lemma \ref{CompCokerKilled}, Proposition \ref{prop513}, and the introductory remark to it. We should remark that no finiteness conditions are not imposed on $I_0$ in Proposition \ref{prop513}.} One should keep in mind that $(R,I)$ is a basic setup; $R$ is a ring and $I$ is its idempotent ideal. Before going further, we clarify the relationship between ``almost integrality" and ``$I$-almost integrality" in a special case.

\begin{lemma}
\label{lemma117}
Let $A_0$ be a ring with a sequence of nonzero divisors $\{t_n\}_{n\geq 0}$ such that for every $n\geq 0$ we have $t_{n+1}^{k_n}=t_nu_n$ for some integer $k_n \ge 2$ and some unit $u_n\in A_0^\times$. Put $A:=A_0[\frac{1}{t_0}]$. Then $a\in (A_0)^*_A$ implies that $t_na\in (A_0)^+_A$ for every $n\geq 0$. Moreover, the converse holds true if $(A_0, (t_0))$ is preuniform. 
\end{lemma}

\begin{proof}
Let us choose an element $a\in A$ and let $n>0$ be an integer. Put $c_n:=\prod_{0\leq i\leq n-1}k_i$. Then by assumption, $t_n^{c_n}=t_0v_n$  for some unit $v_n\in A_0^\times$. Suppose that $a\in (A_0)^*_A$. Then there exists some $d>0$ such that $t_0^{d}a^{m}\in A_{0}$ for every $m\geq 1$. Thus, we have $(t_na)^{c_nd}\in A_0$ and therefore, $t_na\in (A_0)^+_A$. This yields the first assertion. To show the converse, we suppose that $(A_0, (t_0))$ is preuniform and $t_ma \in (A_0)^+_A$ for every $m\geq 1$. Then we have $t_0a^{c_n}=(t_na)^{c_n}v_n^{-1}\in (A_0)^+_A$. Thus, since $c_n\to \infty\ (n\to \infty)$, it is easy to check $t_0a^m\in (A_0)^+_A$ for every $m\geq 1$. Here by assumption, there exists some $e>0$ for which $t_0^e(A_0)_A^+\subset A_0$. Hence we have $a\in (A_0)^*_A$, as wanted. 
\end{proof}

\begin{corollary}
\label{cor128}
Let $(R, I)$ be a basic setup and let $A_0$ be an $R$-algebra. Assume that $I$ admits a sequence of elements $\{t_n\}_{n\geq0}$ such that $A_0$ is $t_0$-torsion free, $IA_0$ is generated by the set of all $t_n\in A_0$ $(n\geq 0)$ and for every $n\geq 0$ we have $t_{n+1}^{k_n}=t_nu_n$ for some integer $k_n \ge 2$ and some unit $u_n\in A_0^\times$. Put $A:=A_0[\frac{1}{t_0}]$. Then the following assertions hold. 
\begin{enumerate}
\item
Let $A'_0\subset A$ be a subring containing $(A_0)^+_A$. If one has $A'_0\subset (A_0)^*_A$, then the inclusion map $(A_0)^+_A\hookrightarrow A'_0$ is an $I$-almost isomorphism (in particular, $(A_0)^+_A\hookrightarrow (A_0)^*_A$ is an $I$-almost isomorphism). 
Moreover, the converse holds true if $(A_0, (t_0))$ is preuniform. 
\item
Equip $A$ with a linear topology so that $A$ is the Tate ring associated to $(A_0, (t_0))$. 
Let ${A}\to {B}$ be a module-finite extension of Tate rings as in Lemma \ref{tatefinex}. 
Then one has $(A_0)^+_{B}\subset {B}^\circ$ and the inclusion map is an $I$-almost isomorphism. 
\end{enumerate}
\end{corollary}

\begin{proof}
The assertion $(1)$ follows from Lemma \ref{lemma117} immediately. 
To show $(2)$, take a ring of definition ${B}_0\subset {B}$ as in Lemma \ref{tatefinex}. Then ${B}^\circ=({B}_0)^*_{B}$ and $(A_0)^+_{B}=({B}_0)^+_{B}$, because ${B}_0$ is integral over $A_0$. Hence the assertion follows from $(1)$. 
\end{proof}

In the situation of Corollary \ref{cor128}, one can obtain an ``almost" variant of Corollary \ref{Bhlem}.

\begin{lemma}
\label{AlmIntComp}
Keep the notation and the assumption as in Corollary \ref{cor128}. 
Denote by $\widehat{A_0}$ the $t_0$-adic completion of $A_0$ and put $A':=\widehat{A_0}[\frac{1}{t_0}]$. 
Then the following conditions are equivalent. 
\begin{itemize}
\item[(a)]
The inclusion map $A_0\hookrightarrow (A_0)^+_A$ (resp.\ $A_0\hookrightarrow (A_0)^*_A$) is an $I$-almost isomorphism.
\item[(b)]
The inclusion map $\widehat{A_0}\hookrightarrow (\widehat{A_0})^+_{A'}$ (resp.\ $\widehat{A_0}\hookrightarrow (\widehat{A_0})^*_{A'}$) is an $I$-almost isomorphism. 
\end{itemize}
\end{lemma}

\begin{proof}
The proof of the lemma follows immediately from Proposition \ref{CICprop} and Corollary \ref{cor128}, because the sequence of ideals $\{t_nA_0\}_{n\geq 0}$ defines the same adic topology on $A_0$. 
\end{proof}

Moreover, notice the following.

\begin{lemma}[{\cite[Lemma 5.3(i)]{Sch12}}]\label{Sch12Lem5.3(i)}
Keep the notation and the assumption as in Corollary \ref{cor128}. Let $M_0$ be an $I$-almost flat $A_0$-module. Then $(M_0^a)_*$ is $t_0$-torsion free, where
$(M_0^a)_*:=\Hom_R(I, M_0)$ (see \cite[2.2.10]{GR03} for this notation).
\end{lemma}

\begin{proof}
By the assumption on $I$ and $A_0$, any $I$-almost zero element in $\Hom_R(I, M_0)$ must be zero. Thus the argument in the proof of \cite[Lemma 5.3(i)]{Sch12} still works under our setting. 
\end{proof}

Next we prove a key result in this section. This is an important consequence of studies in $\S\ref{StudPreUniPair}$.

\begin{theorem}
\label{lemalprjct}
Let $(R, I)$ be a basic setup and let $f_0: A_0\to B_0$ be an $R$-algebra homomorphism with an element $t\in A_0$. Denote by $\widehat{A_0}$ and $\widehat{B_0}$ the $t$-adic completions of $A_0$ and $B_0$, respectively. Let $\widehat{f_0}: \widehat{A_0}\to \widehat{B_0}$ be the $R$-algebra homomorphism induced by $f_0$. Assume that the morphism of pairs $f_0: (A_0, (t))\to (B_0, (t))$ satisfies the condition $(*)$. Then the following assertions hold. 
\begin{enumerate}
\item
The following conditions are equivalent. 
\begin{itemize}
\item[$(a)$]
$B_0$ is $I$-almost finitely generated and $I$-almost projective over $A_0$. 
\item[$(b)$]
$\widehat{B_0}$ is $I$-almost finitely generated and $I$-almost projective over $\widehat{A_0}$. 
\end{itemize}
\item
The following conditions are equivalent. 

\begin{itemize}
\item[$(a)$]
$f_0: A_0\to B_0$ is $I$-almost finite \'etale. 
\item[$(b)$]
$\widehat{f_0}: \widehat{A_0}\to \widehat{B_0}$ is $I$-almost finite \'etale. 
\end{itemize}
\end{enumerate}
\end{theorem}

To prove this theorem, we need the following lemma.

\begin{lemma}\label{LemProjComp}
Keep the notation as in Theorem \ref{lemalprjct}. Assume that $f_0: (A_0, (t))\to (B_0, (t))$ satisfies the condition $(*)$. Assume further that for every $\varepsilon\in I$ and for every integer $n>0$ there exist a finite free $A_0$-module $F$ and $A_0$-homomorphisms $B_0 \to F \to B_0$ whose composition is congruent to multiplication by $\varepsilon$ modulo $(t^n)$. Then $B_0$ is $I$-almost finitely generated and $I$-almost projective over $A_0$.
\end{lemma}

\begin{proof}[Proof of Lemma \ref{LemProjComp}]
Fix $\varepsilon\in I$ arbitrarily. 
By Proposition \ref{cor1416}(2), there exist an integer $l>0$, a finite free $A_0$-module $F_1$, and $A_0$-linear maps $\pi_l: F_1\to
B_0$, $s_l: B_0\to F_1$ such that $\pi_l\circ s_l$ is multiplication by $t^l$. On the other hand, by assumption we have $A_0$-linear maps $\pi_{\varepsilon, l}: F_2\to B_0$ and $s_{\varepsilon, l}: B_0\to F_2$ for some finite free $A_0$-module $F_2$, such that $\pi_{\varepsilon, l}\circ s_{\varepsilon, l}$ is congruent to multiplication by $\varepsilon$ modulo $(t^l)$. 
Thus one can define the $A_0$-linear maps 
$$
f_\varepsilon: B_0\to B_0,\ x\mapsto \frac{1}{t^l}\big(\varepsilon x-(\pi_{\varepsilon, l}\circ s_{\varepsilon, l})(x)\big) 
$$
and
$$
s_\varepsilon: B_0\to F_1\oplus F_2,\ x\mapsto \big((s_l\circ f_\varepsilon)(x), s_{\varepsilon, l}(x)\big).
$$ 
Consider the $A_0$-linear map
$$
\pi_\varepsilon: F_1\oplus F_2\to B_0,\ (a_1, a_2)\mapsto \pi_l(a_1)+\pi_{\varepsilon,l}(a_2).
$$
Then $\pi_\varepsilon\circ s_\varepsilon$ is multiplication by $\varepsilon$. Hence the assertion follows. 
\end{proof}

Now let us complete the proof of Theorem \ref{lemalprjct}.

\begin{proof}[Proof of Theorem \ref{lemalprjct}]
Put $A:=A_0[\frac{1}{t}]$, $A':=\widehat{A_0}[\frac{1}{t}]$, $B:=B_0[\frac{1}{t}]$, and $B':=\widehat{B_0}[\frac{1}{t}]$. Notice that the morphism $\widehat{f_0}: (\widehat{A_0}, (t))\to (\widehat{B_0}, (t))$ satisfies the condition $(*)$ by Theorem \ref{maincor1}. In particular, $B'$ is finite \'etale over $A'$. 

$(1)$: We first show $(a)\Rightarrow (b)$. 
Assume that $(a)$ is satisfied. Then for an arbitrary $\varepsilon\in I$, there exist $A_0$-linear maps $\pi_{\varepsilon}: A_0^{\oplus d}\to B_0$ and $s_\varepsilon: B_0\to A_0^{\oplus d}$ such that $\pi_{\varepsilon}\circ s_\varepsilon$ is multiplication by $\varepsilon$. Let $\widehat{\pi}_{\varepsilon}: (\widehat{A_0})^{\oplus d}\to \widehat{B_0}$ and $\widehat{s}_\varepsilon: \widehat{B_0}\to (\widehat{A_0})^{\oplus d}$ be the induced $\widehat{A_0}$-linear maps. Then we have the natural commutative diagram of $A_0$-linear maps: 
$$
\begin{CD}
B_0 @>{s_\varepsilon}>> A_0^{\oplus d} @>{\pi_\varepsilon}>> B_0\\
@VVV @VVV @VVV\\
\widehat{B_0} @>>{\widehat{s_\varepsilon}}> (\widehat{A_0})^{\oplus d} @>>{\widehat{\pi_\varepsilon}}>\ \widehat{B_0}\ ,
\end{CD}
$$
where the vertical maps become isomorphisms after base extension along $A_0\to A_0/t^nA_0$ for an arbitrary $n>0$. 
Hence $\widehat{\pi_{\varepsilon}}\circ \widehat{s_\varepsilon}$ is multiplication by $\varepsilon$ modulo $(t^n)$. 
Thus, we can apply Lemma \ref{LemProjComp} to this situation. Therefore $(b)$ is satisfied. 

Next we show $(b)\Rightarrow (a)$. Assume that $(b)$ is satisfied. Like the proof of the inverse implication, it suffices to construct $A_0$-linear maps $B_0\to A_0^{\oplus d}$ and $A_0^{\oplus d}\to B_0$ whose composition $B_0\to B_0$ is multiplication by $\varepsilon$ modulo $(t^n)$ for an arbitrary $\varepsilon\in I$ and an arbitrary $n>0$. 
By assumption, there exist $\widehat{A_0}$-linear maps $\widehat{\pi}_{\varepsilon}: (\widehat{A_0})^{\oplus d}\to \widehat{B}$ and $\widehat{s}_\varepsilon: \widehat{B_0}\to (\widehat{A_0})^{\oplus d}$ such that $\widehat{\pi}_{\varepsilon}\circ \widehat{s}_\varepsilon$ is multiplication by $\varepsilon$. 
These maps induce $A_0$-linear maps $\overline{\pi}_\varepsilon: A_0^{\oplus d}/(t^n)\to B_0/(t^n)$ and $\overline{s}_\varepsilon: B_0/(t^n)\to A_0^{\oplus d}/(t^n)$ such that $\overline{\pi}_\varepsilon\circ \overline{s}_\varepsilon$ is multiplication by $\varepsilon$. Since there exists an $A_0$-linear map $\pi_\varepsilon: A_0^{\oplus d}\to B_0$ that is a lift of $\overline{\pi}_\varepsilon$, we are reduced to constructing an $A_0$-linear map $B_0\to A_0^{\oplus d}$ that is a lift of $\overline{s}_\varepsilon$. 
By inverting $t$, $\widehat{s}_\varepsilon$ is extended to a $A'$-linear map $\widehat{s}'_\varepsilon: B'\to A'^{\oplus d}$. 
Since $B'$ is finite \'etale over $A'$, there exist $\widehat{b}_1,\ldots,\widehat{b}_d\in B'$ such that $\widehat{s}'_\varepsilon(x)=(\textnormal{Tr}_{B'/A'}(\widehat{b}_{1}x),\ldots,\textnormal{Tr}_{B'/A'}(\widehat{b}_{d}x))$ for every $x\in B'$ 
(and thus, for each $i=1,\ldots, d$, $\textnormal{Tr}_{B'/A'}(\widehat{b}_{i}x)\in \widehat{A_0}$ if $x\in \widehat{B_0}$).  
For an integer $k>0$, we take $b_1,\ldots, b_d\in B$ such that $\widehat{b}_i\equiv b_i\mod t^{2k}\widehat{B_0}$ for each $i$ (cf.\ Remark \ref{rmk1226}). 
Pick $m\in B_0$. Then 
$$
\textnormal{Tr}_{B'/A'}(b_im)=\textnormal{Tr}_{B'/A'}(\widehat{b}_im)+t^k(\textnormal{Tr}_{B'/A'}(t^kx_i))\ \ \ (i=1,\ldots, d)
$$ 
for some $x_i\in \widehat{B_0}$. Now by Proposition \ref{cor1416}(1), we have $\textnormal{Tr}_{B'/A'}(t^{l}\widehat{B_0})\subset\widehat{A_0}$ for some $l>0$. Thus we may assume that $\textnormal{Tr}_{B'/A'}(t^kx_i)\in \widehat{A_0}$ $(i=1,\ldots, d)$ by increasing $k$ if necessary. Hence we have

\begin{equation}\label{eqtrpl}
\textnormal{Tr}_{B'/A'}(b_im)\equiv \textnormal{Tr}_{B'/A'}(\widehat{b_i}m)\mod t^k\widehat{A_0}\ \ \ (i=1,\ldots, d).
\end{equation}
In particular, $\textnormal{Tr}_{B'/A'}(b_im)\in \widehat{A_0}$. 
Meanwhile, since $B\otimes_AA'\cong B'$ by Theorem \ref{maincor1}, the diagram of $B$-linear maps 
$$
\begin{CD}
B @>\textnormal{Tr}_{B/A}>> A \\
@VVV @VVV \\
B' @>>\textnormal{Tr}_{B'/A'}> A'
\end{CD}
$$
commutes. Hence $\textnormal{Tr}_{B/A}(b_im)\in A_0$ in view of Lemma \ref{Beauville-Laszlo}. Therefore, one can define the $A_0$-linear map 
$$
s_\varepsilon: B_0\to A_0^{\oplus d},\ m\mapsto \big(\textnormal{Tr}_{B/A}(b_1m),\ldots, \textnormal{Tr}_{B/A}(b_dm)\big),
$$
which is a lift of $\overline{s}_\varepsilon$ in view of (\ref{eqtrpl}). 

$(2)$: First we assume that $(a)$ is satisfied. Then $\widehat{B_0}$ is an $I$-almost finitely generated projective $\widehat{A_0}$-module in view of the assertion (1). Hence by \cite[Remark 2.4.12 and Proposition 2.4.18]{GR03}, it follows that $\widehat{B_0}$ is an  $I$-almost flat and $I$-almost finitely presented $\widehat{A_0}$-module. On the other hand, since $A' \to B'$ is unramified by Theorem \ref{maincor1} and $\widehat{A_0}/(t) \to \widehat{B_0}/(t)$ is $I$-almost unramified by \cite[Lemma 3.1.2]{GR03}, it follows from \cite[Theorem 5.2.12]{GR03}, together with the fact that $\widehat{B_0}$ is an $I$-almost finitely presented $\widehat{A_0}$-module, that $\widehat{A_0}\to \widehat{B_0}$ is $I$-almost unramified. Thus we find that $(b)$ is satisfied. The inverse implication $(b)\Rightarrow (a)$ can be shown similarly. 
\end{proof}

\subsection{Proof of the almost purity theorem}

Now we are ready to prove the almost purity theorem by Davis and Kedlaya. In this subsection, we fix a prime number $p>0$, and for any ring $R$, we denote by $\widehat{R}$ the $p$-adic completion of $R$. Moreover for a ring map $f: R\to S$, we denote by $\widehat{f}$ the ring map $\widehat{R}\to \widehat{S}$ induced by $f$.

\begin{theorem}[Almost purity]
\label{almostpurity}
Let $A_0$ be a $p$-torsion free Witt-perfect ring and let $f_0: A_0\to B_0$ be a ring map. Put $I:=\sqrt{pA_0}$. Assume that the morphism $f_0: (A_0, (p))\to (B_0, (p))$ satisfies the condition $(*)$, $A_0$ is integrally closed in $A_0[\frac{1}{p}]$, $(A_0)^+_{B_0[\frac{1}{p}]}\subset B_0$, and $(A_0, I)$ is a basic setup\footnote{For example, this assumption is realized if $A_0$ admits $\{p^\frac{1}{p^n}\}_{n\geq 0}$, $A_0$ is $p$-adically Zariskian, or $A_0$ is an algebra over a $p$-torsion free Witt-perfect valuation domain of rank $1$ (cf.\ Example \ref{BasicSetupEx} and Lemma \ref{zarsemiperf}). }. Then the following assertions hold.
\begin{enumerate}
\item
$B_0$ is also Witt-perfect.
\item
$f_0: A_0\to B_0$ is $I$-almost finite \'etale.
\end{enumerate}
\end{theorem}

\begin{proof}
Let $\mathcal{A}$ and  $\mathcal{B}$ be Banach rings associated to complete Tate rings $(\widehat{A_0}, (p))$ and $(\widehat{B_0}, (p))$ respectively such that the map $\mathcal{A}\to \mathcal{B}$ induced by $\widehat{f_0}$ is bounded. Then by Proposition \ref{decompperf} and Corollary 
\ref{DecompTopNil}, $\mathcal{A}$ is perfectoid and $\mathcal{A}^{\circ\circ}=I\widehat{A_0}$. Moreover, since $(\widehat{A_0}, (p))\to(\widehat{B_0}, (p))$ satisfies the condition $(*)$ by Theorem \ref{maincor1}, $\mathcal{A}\to \mathcal{B}$ is finite \'etale, $\widehat{B_0}\subset (\widehat{A_0})^*_\mathcal{B}$, and $(\widehat{A_0})^*_\mathcal{B}=(\widehat{B_0})^*_\mathcal{B}$ by Proposition \ref{cor1416}(3). 
Meanwhile, we also have $p(B_0)^*_B=p(A_0)^*_B\subset (A_0)^+_B\subset B_0$ by Lemma \ref{AlmIntMod-Fin} and Proposition \ref{cor1416}(3), which implies that $p(\widehat{B_0})^*_\mathcal{B}\subset \widehat{B_0}$ by Proposition \ref{CICprop}. 
Thus we find that 
$$
p\widehat{B_0}\subset p(\widehat{A_0})^*_\mathcal{B}=p(\widehat{B_0})^*_\mathcal{B}\subset \widehat{B_0}. 
$$ 
Hence, the topology on $\mathcal{B}$ coincides with the canonical topology on 
$\widehat{B_0}[\frac{1}{p}]$ as a finitely generated $\mathcal{A}$-module by Corollary \ref{cor1547205}(1). Thus, in view of Corollary \ref{cor128}, the maps $\widehat{A_0}\hookrightarrow \mathcal{A}^\circ$ and $\widehat{B_0}\hookrightarrow \mathcal{B}^\circ$ are $I$-almost isomorphisms. Moreover, by almost purity theorem by Kedlaya-Liu \cite[Theorem 3.6.21 and 5.5.9]{KL15}, it follows that $\mathcal{A}^\circ \to \mathcal{B}^\circ$ is $I$-almost finite \'etale (and therefore so is $\widehat{f_0}$) and $\mathcal{B}$ is perfectoid. 
In particular, $B_0$ is Witt-perfect by Lemma \ref{lem2231930}. 
Now the assertion $(2)$ follows from Theorem \ref{lemalprjct}. 
\end{proof}

\begin{remark}
Resume the notation of Theorem \ref{almostpurity}. Here is a way to check almost flatness for the extension $A_0 \to B_0$. Suppose that $A_0[\frac{1}{p}] \to B_0[\frac{1}{p}]$ is flat and $A_0/(p) \to B_0/(p)$ is $I$-almost flat. Then since $p$ is a nonzero divisor on both $A_0$ and $B_0$, a simple discussion using the short exact sequence: $0 \to A_0 \xrightarrow{p} A_0 \to A_0/(p) \to 0$ shows that $\Tor^{A_0}_i(B_0,A_0/(p))=0$ for $i>0$. By applying \cite[Lemma 5.2.1]{GR03} (see also \cite[Lemma 1.2.5]{F92} for the absolute version), one sees that $B_0$ is an $I$-almost flat $A_0$-module.
\end{remark}

In \cite{NS19}, we apply Theorem \ref{almostpurity} to construct big (almost) Cohen-Macaulay algebras with distinguished properties. 
Here we illustrate a key step to it in a simple situation. 

\begin{example}
Let $R$ be a Noetherian complete local domain of mixed characteristic $p>0$ with perfect residue field $k$. Let $p, x_{2},\ldots, x_{d}$ be a system of parameters in $R$. 
Then there exists a module-finite extension $\iota: W(k)[[x_{2},\ldots, x_{d}]]\hookrightarrow R$ by Cohen's structure theorem. 
Since the induced field extension $\textnormal{Frac}(W(k)[[x_{2},\ldots, x_{d}]])\hookrightarrow \textnormal{Frac}(R)$ is finite separable, there exists an element $g\in W(k)[[x_{2},\ldots, x_{d}]]\setminus (p)$ such that 
$\iota$ becomes finite \'etale after inverting $pg$. Here, we assume that $g=1$ for simplicity (\cite{NS19} deals with the general cases). Set  
$$
A_{0}:=\bigcup_{n>0}W(k)[[x_{2},\ldots, x_{d}]][p^{\frac{1}{p^n}}, x_{2}^\frac{1}{p^n},\ldots, x_{d}^\frac{1}{p^n}]. 
$$ 
Then $A_{0}$ is a Witt-perfect ring. Moreover, by base change, we obtain a finite \'etale ring map 
$$
f: A_{0}[\frac{1}{p}]\to (A_{0}\otimes_{W(k)[[x_{2},\ldots, x_{d}]]}R)[\frac{1}{p}]. 
$$
Let $B_{0}$ be the integral closure of $A_{0}$ in $(A_{0}\otimes_{W(k)[[x_{2},\ldots, x_{d}]]}R)[\frac{1}{p}]$. Then by Theorem \ref{almostpurity}, the induced map $A_{0}\to B_{0}$ is $(p)^{\frac{1}{p^{\infty}}}$-almost finite \'etale, and $B_{0}$ is Witt-perfect. In particular, $A_{0} \to B_{0}$ is $(p)^{\frac{1}{p^{\infty}}}$-almost flat. Hence $B_{0}$ is an almost Cohen-Macaulay $R$-algebra with respect to $(p, x_{2},\ldots. x_{d})$ in the sense of \cite[Definition 5.4]{Sh18}. Notice that the map $R\to B_{0}$ is integral. 
\end{example}

Let $(R, I)$ be a basic setup and let $A$ be an $R$-algebra. 
Let us denote by ${\bf{Alg}}^a(A)$ the quotient category of the category of $A$-algebras ${\bf{Alg}}(A)$ by the Serre subcategory of objects in ${\bf{Alg}}(A)$ that are $I$-almost zero. Let $(~)_*:{\bf{Alg}}^a(A) \to {\bf{Alg}}(A)$ be the \textit{functor of almost elements} (see \cite[2.2.10]{GR03}). We define ${\bf{F.Et}}(A)$ to be the category of finite \'etale $A$-algebras. We also define ${\bf{F.Et}}^a(A)$ to be the full subcategory of ${\bf{Alg}}^a(A)$ that consist of $I$-almost finite \'etale $A$-algebras.

\begin{corollary}
Let $A_0$ be a $p$-torsion free algebra over a $p$-torsion free Witt-perfect valuation domain $V$ of rank $1$. Assume that $A_0$ is Witt-perfect and integrally closed in $A_0[\frac{1}{p}]$. Put $I:=\sqrt{pA_0}$. Consider the basic setup $(A_0, I)$ and the diagram of functors: 
\begin{equation}\label{BLFEtPerf}
\begin{CD}
{\bf{F.Et}}^a(A_0)@>\Phi_1>> {\bf{F.Et}}^a(\widehat{A_0}) \\
@V\Phi_2VV @VV\Phi_4V \\
{\bf{F.Et}}(A_0[\frac{1}{p}]) @>>\Phi_3>\ {\bf{F.Et}}(\widehat{A_0}[\frac{1}{p}])
\end{CD}
\end{equation}
where $\Phi_1$, $\Phi_2$, $\Phi_3$ and $\Phi_4$ are given by the association $B_0\mapsto B_0\otimes_{A_0^a}(\widehat{A_0})^a$, $B_0\mapsto (B_0)_*[\frac{1}{p}]$, $B\mapsto B\otimes_{A_0[\frac{1}{p}]}\widehat{A_0}[\frac{1}{p}]$, and $B'_0\mapsto (B'_0)_*[\frac{1}{p}]$, respectively (cf.\ \cite[Lemma 3.1.2(i)]{GR03}). Then the following assertions hold. 
\begin{enumerate}
\item
$\Phi_2$ and $\Phi_4$ yield equivalences of categories. Moreover, 
for an $I$-almost finite \'etale $A_0$-algebra $C_0$, the natural $A_0$-algebra homomorphism $C_0\otimes_{A_0}\widehat{A_0}\to \widehat{C_0}$ is an $I$-almost isomorphism. 
\item
Assume that $A_0$ is henselian along the ideal $(p)$. Then $\Phi_1$ and $\Phi_3$ yield equivalences of categories. 
\end{enumerate}
\end{corollary}

\begin{proof}
We can equip $\widehat{V}[\frac{1}{p}]$ and $\widehat{A_0}[\frac{1}{p}]$ with norms so that $\widehat{V}[\frac{1}{p}]$ is a perfectoid field and $\widehat{A_0}[\frac{1}{p}]$ is a perfectoid $\widehat{V}[\frac{1}{p}]$-algebra by Proposotion \ref{perfalgdecomp}. 
Thus by the almost purity theorem for perfectoid $\widehat{V}[\frac{1}{p}]$-algebras \cite[Theorem 4.17, Theorem 5.2, Proposition 5.22, and Theorem 7.9]{Sch12} and Proposition \ref{cor1416}(3), $\Phi_4$ admits a quasi-inverse $\Psi_4: {\bf{F.Et}}(\widehat{A_0}[\frac{1}{p}])\to {\bf{F.Et}}^a(\widehat{A_0})$ given by the association $B'\mapsto (\widehat{A_0})^{*a}_{B'}$. 
More precisely, in view of \cite[Lemma 5.6]{Sch12}, for any $B'_0\in \textnormal{ob}({\bf{F.Et}}^a(\widehat{A_0}))$ we have 
\begin{equation}\label{CatEqPerfTateInt}
(\widehat{A_0})^*_{\Phi_4(B'_0)}=(B'_0)_*.
\end{equation}
On the other hand, by Theorem \ref{almostpurity}, we can also define a functor $\Psi_2: {\bf{F.Et}}(A_0[\frac{1}{p}])\to {\bf{F.Et}}^a(A_0)$ given by the association $B\mapsto (A_0)^{*a}_B$. 
Let us show that $\Psi_2$ gives a quasi-inverse of $\Phi_2$. 
The nontrivial part is to prove the existence of a functorial isomorphism $B_0\cong (\Psi_2\circ\Phi_2)(B_0)$ for $B_0\in\textnormal{ob}({\bf{F.Et}}^a(A_0))$. 
We let $C_0$ be the $I$-almost finite \'etale $A_0$-algebra $(B_0)_*$. 
Then $C_0$ is $p$-torsion free by Lemma \ref{Sch12Lem5.3(i)}, and thus $\widehat{A_0}\to \widehat{C_0}$ is also $I$-almost finite \'etale by Theorem \ref{lemalprjct}. 
Put $C:=C_0[\frac{1}{p}]$ and $C':=\widehat{C_0}[\frac{1}{p}]$. 
Since $C'=\Phi_4((\widehat{C_0})^a)$, we have 
$(\widehat{C_0})^*_{C'}=(\widehat{A_0})^*_{C'}=((\widehat{C_0})^a)_*$ by Proposition \ref{cor1416}(3) and (\ref{CatEqPerfTateInt}). 
Therefore the inclusion map $\widehat{C_0}\hookrightarrow (\widehat{C_0})^*_{C'}$ is an $I$-almost isomorphism. Hence by Lemma \ref{AlmIntComp}, $C_0\hookrightarrow (C_0)^*_C$ is also an $I$-almost isomorphism. Applying the functor of almostification $(\ \cdot\ )^a$, we obtain an isomorphism $C^a_0 \cong (\Psi_2\circ\Phi_2)(B_0)$, which yields the desired isomorphism. Next we apply the functor $\Phi_4\circ (\ \cdot\ )^a$ to the natural map $\varphi: C_0\otimes_{A_0}\widehat{A_0}\to \widehat{C_0}$. 
Then we obtain the natural map $(C_0\otimes_{A_0}\widehat{A_0})[\frac{1}{p}]\to\widehat{C_0}[\frac{1}{p}]$, which is an isomorphism by Theorem \ref{maincor1}. Thus, since $\Phi_4$ is fully faithful, we find that $\varphi$ is an $I$-almost isomorphism. Hence the assertion $(1)$ follows. If further $A_0$ is henselian along the ideal $(p)$, then $\Phi_3$ yields an equivalence of categories due to \cite[Proposition 5.4.53]{GR03}, and therefore $(1)$ implies that $\Phi_1$ yields an equivalence of categories (because (\ref{BLFEtPerf}) is essentially commutative). It completes the proof. 
\end{proof}

\section{Appendix: A historical remark on the (almost) purity theorem}

In this appendix, we will provide some background history around the almost purity theorem as well as its classical version, which is known as the \textit{Purity over Noetherian local rings}. Let us begin with the definition of purity for schemes.

\begin{definition}[Purity]
Let $X$ be a scheme together with an open subset $U \subset X$. Let ${\bf{Et}}(Y)$ denote the category of \'etale finite $Y$-schemes. If the restriction functor:
$$
{\bf{F.Et}}(X) \to {\bf{F.Et}}(U);~ Y \mapsto Y':=Y \times_X U
$$
is an equivalence of categories, then we say that $(X,U)$ is a \textit{pure pair}.
\end{definition}

To make this definition work, one is requited to put more conditions, such as normality. Let us recall the following classical result due to Grothendieck.

\begin{theorem}[Grothendieck]
Assume that $(R,\fm)$ is a Noetherian local ring and let $X:=\Spec(R)$ and $U:=\Spec(R) \setminus \{\fm\}$. Then the following assertions hold:
\begin{enumerate}
\item
If $R$ is a regular local ring with $\dim R \ge 2$, then $(X,U)$ is pure.

\item
If $R$ is a complete intersection with $\dim R \ge 3$, then $(X,U)$ is pure.
\end{enumerate}
\end{theorem}

\begin{proof}
The proof of the first statement is in \cite[Tag 0BMA]{Stacks}, while
the second statement is in \cite[Tag 0BPD]{Stacks}.
\end{proof}

These results are of use to the proof of the so-called \textit{Zariski-Nagata's purity theorem} whose proof is found in \cite[Tag 0BMB]{Stacks}.

\begin{theorem}[Purity of branch locus]
Assume that $f:X \to S$ is a finite dominant morphism of Noetherian integral schemes such that $S$ is regular and $X$ is normal. Then the ramification locus of $f$ is of pure codimension one.
\end{theorem}

The almost analogue of pure pair has been suggested in Gabber-Ramero's treatise \cite[Definition 14.4.1]{GR18}, in which case the almost purity theorem takes the form of $X=\Spec(R)$ and $U=\Spec (R[\frac{1}{p}])$ for a certain big ring $R$. To the best of authors' knowledge, the first appearance of the almost purity theorem is Tate's work on $p$-divisible groups over local fields \cite{Tate67} and its higher-dimensional analog was studied by Faltings \cite{Fa88}, who also outlined ideas of almost ring theory. Faltings gave another proof in \cite{Fa02}, where he made an effective use of the Frobenius action on certain local cohomology modules combined with his normalized length. We refer the reader to Olsson's notes \cite{Ol09}. An ultimate version of the almost purity was proved by Scholze in \cite{Sch12}, using adic spaces. We remark that Faltings started with an essentially smooth algebra over a discrete valuation ring and then constructed a huge ring extension on which the Frobenius map has good behavior. His proof required complicated ramification theory. On the other hand, Scholze started with a ring that is already big enough so that he avoided ramification theory.

\begin{acknowledgement}
The authors are grateful to Professor K. Fujiwara for encouragement and comments on this paper. Our gratitude also goes to some anonymous referee for reading the paper thoroughly and providing many constructive comments that has led to improve the presentation of the present article. Kazuma Shimomoto was partially supported by JSPS Grant-in-Aid for Scientific Research(C) 18K03257.
\end{acknowledgement}

\end{document}